\documentclass[11pt]{article}
\usepackage{amsmath,amsthm,amssymb}
\usepackage[mathscr]{euscript}
\usepackage{verbatim}
\usepackage{dsfont,bm}
\usepackage{stmaryrd} 
\usepackage{color}
\usepackage{mathtools}
\usepackage{graphicx}
\usepackage{enumerate}
\usepackage{esint} 
\usepackage{mathabx}
\allowdisplaybreaks[4]

\usepackage{xcolor}
\definecolor{darkgreen}{rgb}{0,0.75,0}
\definecolor{darkred}{rgb}{0.75,0,0}
\definecolor{darkmagenta}{rgb}{0.5,0,0.5}
\usepackage[colorlinks,citecolor=darkgreen,linkcolor=darkred,urlcolor=darkmagenta]{hyperref}
\newtheorem{theorem}{Theorem}[section]
\newtheorem{lem}[theorem]{Lemma}
\newtheorem{lemma}[theorem]{Lemma}
\newtheorem{prop}[theorem]{Proposition}

\theoremstyle{definition}
\newtheorem{definition}[theorem]{Definition}

\newtheorem{remark}[theorem]{Remark}

\numberwithin{equation}{section}

\def\be{\begin{equation}}
	\def\ee{\end{equation}}
\def\bes{\begin{equation*}}
	\def\ees{\end{equation*}}

\newcommand{\arxiv}[1]{{\tt \href{http://arxiv.org/abs/#1}{arXiv:#1}}}
\newcommand{\set}[1]{\left\{ #1 \right\}}

\newcommand{\abs}[1]{{\left\lvert #1\right\rvert}}
\newcommand\norm[1]{\left\lVert #1\right\rVert} 
\newcommand{\one}{\mathds{1}} 

\newcommand{\diag}[0]{\operatorname{diag}}

\newcommand{\Capa}{\operatorname{Cap}}

\def\ol{\overline}           

\def\eps{\varepsilon}

\def\to {\rightarrow}

\def\q{\quad} 
\def\dint{\int\kern-.6em\int}

\newcommand\restr[2]{{
		\left.\kern-\nulldelimiterspace 
		#1 
		\vphantom{\big|} 
		\right|_{#2} 
}} 

\def\diam{{\mathop{{\rm diam }}}}
\def\dist{{\mathop {{\rm dist}}}}

\newcommand{\on}[1]{\operatorname{ #1}}

\def\wt{\widetilde}
\def\wh{\widehat}

\def\be{\begin{equation}}
	\def\ee{\end{equation}}
\def\bes{\begin{equation*}}
	\def\ees{\end{equation*}}
\def\ba{\begin{align}}
	\def\ea{\end{align}}
\def\xxea{\end{align}}
\def\bas{\begin{align*}}
\def\eas{\end{align*}}

\definecolor{dgreen}{rgb}{0, 0.6, 0.1}
\definecolor{dblue}{rgb}{0, 0.0, 0.6}
\definecolor{vdblue}{rgb}{0,.08, 0.45}
\definecolor{dred}{rgb}{0.7, 0.0, 0.0}
\definecolor{vdblue}{rgb}{0,.08, 0.45}

\definecolor{purple}{rgb}{0.6, 0.0, 0.6}
\definecolor{mytext}{rgb}{0.1, 0.1, 0.1}

\def\Cap{\operatorname{Cap}} 

\begin{document}

\font\titlefont=cmbx14 scaled\magstep1
\title{\titlefont A simplified characterization of stable-like heat kernel estimates} 
\author{Mathav Murugan} 
\renewcommand{\thefootnote}{}
\footnotetext{The author is partially supported by NSERC and the Canada research chairs program.}
\renewcommand{\thefootnote}{\arabic{footnote}}
\setcounter{footnote}{0}

\maketitle
\vspace{-0.5cm}
\begin{abstract}
	We study   heat kernel estimates for symmetric pure jump processes on general metric measure spaces.
	Building on recent progress in the local setting due to S.~Eriksson-Bique, we develop a non-local version
	of the Whitney blending technique and use it to relate stable-like heat kernel estimates to   capacity upper bounds.
	Under two-sided stable-like bounds on the jump kernel, we show that a capacity upper bound across annuli
	implies a cutoff Sobolev inequality, and we obtain a characterization of stable-like heat kernel estimates
	in terms of these conditions.
	As a consequence, we give an affirmative answer to a conjecture of A.~Grigor'yan, E.~Hu, and J.~Hu.
\end{abstract}

\section{Introduction}

The behavior of the heat kernel is intimately tied to the geometry of the underlying space.
A landmark result in this direction, independently due to Grigor'yan and Saloff-Coste, shows that Gaussian heat kernel estimates and the parabolic Harnack inequality for Brownian motion on a Riemannian manifold are characterized by two analytic–geometric properties: volume doubling and the Poincaré inequality \cite{Gri,Sal}.
Over the past three decades, analogous characterizations---often requiring substantial modifications in both formulation and proof---have been established in a variety of settings, including diffusions on local Dirichlet spaces \cite{Stu}, random walks on graphs and metric spaces \cite{Del,MS23}, diffusions on fractal-like spaces \cite{BB,BBK,GHL}, jump processes \cite{CKW-hke,CKW-phi,GHH18,MS}, and symmetric Markov processes with both diffusion and jumps \cite{CKW-diff-jump}.
Similar characterizations of the elliptic Harnack inequality (EHI) for diffusions have also been obtained \cite{BM, BCM}. An important application of these characterizations is the stability of heat kernel estimates and related analytic properties such as elliptic and parabolic Harnack inequalities under perturbations of the space or the underlying Dirichlet form and also to establish such properties in examples of interest.

On fractals and many fractal-like spaces, beyond volume doubling
and the Poincar\'e inequality, the \textbf{cutoff Sobolev inequality} and its variants constitute a fundamental analytic tool.
Introduced by Barlow and Bass in the setting of random walks on graphs and subsequently extended to other settings,
this inequality plays a central role in the characterization and stability of heat kernel estimates and Harnack inequalities, such as:
\begin{enumerate}[(i)]
	\item Parabolic Harnack inequality and two-sided sub-Gaussian heat kernel
	estimates for diffusions  \cite[Theorem 1.5]{BB}, \cite[Theorem 2.16]{BBK} \cite[Theorem 1.2]{GHL}.
	\item Elliptic Harnack inequality for diffusions \cite[Theorem 5.15]{BM},\cite[Theorem 7.9]{BCM}.
	\item Sub-Gaussian heat kernel upper bounds for diffusions \cite[Theorem 1.12]{AB}.
	\item Two-sided stable-like heat kernel bounds for pure jump processes \cite[Theorem 1.13]{CKW-hke}, \cite[Theorem 1.12]{GHH18}, \cite[Theorem 1.5]{MS}.
	\item Stable-like heat kernel upper bounds for pure jump processes \cite[Theorem 1.15]{CKW-hke},
	and parabolic Harnack inequality for pure jump processes \cite[Theorem 1.20]{CKW-phi}.
	\item Heat kernel bounds (two-sided and upper) and parabolic Harnack inequality
	for processes with both diffusions and jumps \cite[Theorems 1.13, 1.14 and 1.18]{CKW-diff-jump}.
\end{enumerate}
Although the cutoff Sobolev inequality and their variants are a powerful tool for establishing the stability results mentioned above, it is often difficult to verify in concrete examples. It is therefore natural to seek simpler sufficient conditions, such as an upper bound on capacity across annuli. This motivation underlies a conjecture of Grigor'yan, Hu, and Lau in setting~(i) above  \cite[Conjecture 4.15]{GHL14}, \cite[p.~1493, p.~1498]{GHL}, as well as a related conjecture of Grigor'yan, E.~Hu, and J.~Hu in setting~(iv) \cite[Conjecture~1.14]{GHH18}.
Let us also mention that cutoff Sobolev inequality has proven to be more generally useful beyond   stability results such as obtaining   singularity of energy measures \cite{KM-aop,Yan} and also for obtaining heat kernel estimates for reflected diffusions \cite{Mur24} and diffusions on  Laakso-type spaces \cite{Mur26}.

Recently, the conjecture of Grigor'yan, Hu, and Lau was resolved in setting~(i) (and hence also in setting~(ii))
by S.~Eriksson-Bique in a breakthrough work introducing the technique of \emph{Whitney blending} \cite{Eri}.
In fact, the Whitney blending technique in \cite{Eri} is implemented in the more general framework of
$p$-Dirichlet spaces introduced in \cite[\textsection 2.2]{EM}, which allows for non-linear energies but assumes locality of the energy.
The purpose of this work is to develop the Whitney blending technique in a non-local setting and to give
an affirmative answer to \cite[Conjecture~1.14]{GHH18} in setting~(iv).
More generally, one may optimistically conjecture that the cutoff Sobolev inequality and its variants
can be replaced by simpler capacity upper bounds in all of the characterizations discussed above.
We leave the resolution of analogous conjectures in the remaining settings, such as~(iii),~(v), and~(vi),
as an interesting direction for future research.
We refer to \cite{Ant,Mur23,Yan} for recent partial progress on the Grigor'yan--Hu--Lau conjecture
prior to its resolution in \cite{Eri}.

While Whitney covers have previously been used to obtain heat kernel bounds for diffusion processes, most notably in the analysis of reflected and killed diffusions in \cite{GS}, their use in deriving heat kernel bounds for jump processes is new. This constitutes a key methodological novelty of the present work. Moreover, the Whitney blending approach in the nonlocal setting requires substantially new ingredients not present in \cite{Eri}, including sharp two sided estimates on the jump kernel and the use of truncated Dirichlet forms.

Our main results are as follows:
\begin{enumerate}[(A)]
	\item Two-sided stable-like bounds on the jump kernel together with a capacity upper bound
	imply a cutoff Sobolev inequality (Theorem~\ref{t:main}).
	\item A characterization of stable-like heat kernel estimates in terms of the two-sided bounds
	on the jump kernel and the capacity upper bound above (Theorem~\ref{t:char-hke}).
\end{enumerate}
The result (B) provides an affirmative answer to \cite[Conjecture~1.14]{GHH18}.

\subsection{Framework and results}
Let $(X,d)$ be a complete, locally compact, separable metric space, and let $m$ be a positive Radon measure on $X$ with full support. We always assume that $X$ contains at least two points. Such a triple $(X,d,m)$ is called a \emph{metric measure space}.
We set $B(x,r):=\{y\in X\mid d(x,y)<r\}, \ol{B}(x,r):=\{y \in X \mid d(x,y) \le r\}$, and $V(x,r)= m(B(x,r))$ for $x \in X, r \in (0,\infty)$. 

We consider a \emph{symmetric Dirichlet form} $(\mathcal{E},
\mathcal{F})$ on $L^2(X,m)$. In other words, $\mathcal{F}$ is a dense linear subspace of $L^2(X,m)$, $\mathcal{E}:\mathcal{F} \times \mathcal{F} \to \mathbb{R}$ is symmetric, non-negative definite, bilinear form that is \emph{closed} ($\mathcal{F}$ is a Hilbert space under the inner product $\mathcal{E}_1(\cdot,\cdot)= \mathcal{E}(\cdot,\cdot) + \langle \cdot,\cdot \rangle_{L^2(X,m)}$) and \emph{Markovian} (for any $f \in \mathcal{F}$, we have $\wh f:=(0\vee f)\wedge 1 \in \mathcal{F}$ and $\mathcal{E}(\wh f,\wh f) \le \mathcal{E}(f,f)$). We assume that $(\mathcal{E},\mathcal{F})$ is \emph{regular}; that is, $\mathcal{F}\cap\mathcal{C}_{\mathrm{c}}(X)$ is dense both in $(\mathcal{F},\mathcal{E}_{1})$
and in $(\mathcal{C}_{\mathrm{c}}(X),\|\cdot\|_{\mathrm{sup}})$.  We assume that $(\mathcal{E},\mathcal{F})$ is a \emph{pure jump} type Dirichlet form; that is, there exists a symmetric positive Radon measure on $X \times X \setminus \diag$ such that 
\[
\mathcal{E}(f,f) = \int_{X \times X \setminus \diag} (f(x)-f(y))^2\, J(dx,dy), \q \mbox{for all $f \in \mathcal{F}$,}
\]
where $\diag = \set{(x,x) \mid x \in X}$ denotes the diagonal. The Radon measure $J$ is called the \emph{jump  measure}; cf.~\cite[Theorem 3.2.1]{FOT}.

We assume that 
the measure $m$ satisfies the following \emph{volume doubling} property \hypertarget{vd}{$\operatorname{VD}$}: there exists $C_D >1$ such that
\begin{equation*}  \tag{VD}
	V(x,2r)\le C_{D} V(x,r), \q \mbox{for any $x \in X, r \in (0,\infty)$}.
\end{equation*}
Setting $\alpha =\log_2 C_D$ and iterating  \hyperlink{vd}{$\operatorname{VD}$}, we obtain the following well-known consequence 
\begin{equation}\label{e:vd-iter}
	V(x,s)\le C_D^2 \left( \frac{d(x,y)+s}{r}\right)^\alpha V(y,r), \quad \mbox{for all $x \in X, 0<r<s<\infty$.}
\end{equation}
We recall the closely related reverse volume doubling property and a generalized version introduced in \cite[Definition 1.4]{Mal}.
\begin{definition}
	We say that a metric measure space $(X,d,m)$ satisfies the \emph{reverse volume doubling} property \hypertarget{rvd}{$\operatorname{RVD}$} if there exists $\beta \in (0,\infty), C \in (1,\infty)$ such that for all $x \in X, 0 < r \le R < \diam(X,d)$, we have 
	\begin{equation*}  \tag{RVD}
		\frac{m(B(x,R))}{m(B(x,r)))}\ge C^{-1} \frac{R^\beta}{r^\beta}.
	\end{equation*}
	For any $x \in X$, define $D_x:= \inf_{y \in X \setminus \{x\}} d(x,y)$. We say that $(X,d,m)$ satisfies the \emph{quasi reverse volume doubling} property \hypertarget{qrvd}{$\operatorname{QRVD}$} if there exists $\beta \in (0,\infty), C \in (1,\infty)$ such that for all $x \in X, D_x < r \le R < \diam(X,d)$, we have 
	\begin{equation*}  \tag{QRVD}
		\frac{m(B(x,R))}{m(B(x,r)))}\ge C^{-1} \frac{R^\beta}{r^\beta}.
	\end{equation*}
\end{definition}
Much of the existing work on stable-like heat kernel estimates assumes the reverse volume doubling condition \hyperlink{rvd}{$\operatorname{RVD}$}.
Since many natural examples, such as graphs with polynomial volume growth, do not satisfy \hyperlink{rvd}{$\operatorname{RVD}$} but do satisfy \hyperlink{qrvd}{$\operatorname{QRVD}$},
we work throughout in this more general setting.

We say that    $\phi:[0,\infty) \to [0,\infty)$ is a \emph{scale function} if $\phi$ is a homeomorphism such that there exist constants $C_\phi \ge 1$, $\beta_2 \ge \beta_1 >0$ such that 
\begin{equation}  \label{e:scale}
	C_\phi^{-1} \left(\frac R r\right)^{\beta_1} \le \frac{\phi(R)}{\phi(r)} \le C_\phi \left(\frac R r\right)^{\beta_2} \quad \mbox{for all $0 < r \le R$.}
\end{equation}

\begin{definition} Let  $(X,d,m)$ be a metric measure space and let $(\mathcal{E},\mathcal{F})$ be a pure jump Dirichlet form on $L^2(X,m)$ with jump measure $J$. 
	We say that the two-sided jump kernel bounds
	\hypertarget{jphi}{$\operatorname{J(\phi)}$} hold if there exist a symmetric
	function $J : X \times X \to [0,\infty)$ and constants $c, C > 0$ such that
	\[
	J(dx,dy) = J(x,y)\, m(dx)\, m(dy),
	\]
	and
	\begin{equation*}
		\tag*{$\operatorname{J(\phi)}$}
		\frac{c}{m(B(x,d(x,y)))\, \phi(d(x,y))}
		\le J(x,y)
		\le \frac{C}{m(B(x,d(x,y)))\, \phi(d(x,y))},
	\end{equation*}
	for all $x,y \in X$ with $x \neq y$.
	The above upper and lower bounds on $J$ will be denoted by
	$\operatorname{J(\phi)}_{\le}$ and $\operatorname{J(\phi)}_{\ge}$,
	respectively.

\end{definition}

For the remainder of this work, we always assume that the jump measure $J$ on $X \times X \setminus \diag_X$ is of the form $J(dx,dy)= J(x,y) \,m(dx)\,m(dy)$ where $J: X \times X \to [0,\infty)$ is a non-negative, symmetric, measurable function.  We define the \emph{energy measure} $\Gamma(f,f)$ of a function $f \in \mathcal{F}$ by 
\begin{equation} \label{e:def-gamma} 
	\Gamma(f,f)(A) = \int_{A} \int_X (f(x)-f(y))^2 J(x,y)\,m(dy)\,m(dx).
\end{equation}
We also denote by $\Gamma(f,f): X \to \mathbb{R}$, the dennsity of the measure $\Gamma(f,f)$ with respect to $m$ given by 
\[
\Gamma(f,f)(x):= \int_{X}(f(x)-f(y))^2 J(x,y) \, m(dy).
\]
By an abuse of notation, we use the same symbol for both the measure and its
density with respect to $m$ for $J$ and $\Gamma$; the meaning will be clear
from the context.
\begin{definition}
	\begin{enumerate}[(i)]
		\item
		Let $U \subset V$ be open sets in $X$ with $U \subset \overline{U} \subset V$.
		A \textbf{cutoff function} for $U \subset V$ is a non-negative, bounded,
		measurable function $\phi$ such that $\phi \equiv 1$ on $U$,
		$\phi \equiv 0$ on $V^c$, and $0 \le \phi \le 1$ on $X$.
		
		\item
		We say that the \textbf{cutoff Sobolev inequality}
		\hypertarget{csj}{$\operatorname{CSJ(\phi)}$} holds if there exist constants
		$a_0 \in (0,1]$ and $C_1, C_2 \in (0,\infty)$ such that for all
		$0 < r \le R$, $x \in X$, and all $f \in \mathcal{F}$, there exists a cutoff
		function $\psi \in \mathcal{F}$ for $B(x,R) \subset B(x,R+r)$ such that
		\begin{equation*}
			\tag*{$\operatorname{CSJ(\phi)}$}
			\int_{B(x,R+(1+a_0)r)} f^2\, d\Gamma(\psi,\psi)
			\le C_1 \int_{U \times U^*} (f(x)-f(y))^2\, J(dx,dy)
			+ \frac{C_2}{\phi(r)} \int_{B(x,R+(1+a_0)r)} f^2\, dm,
		\end{equation*}
		where
		\[
		U = B(x,R+r) \setminus B(x,R),
		\qquad
		U^* = B(x,R+(1+a_0)r) \setminus B(x,R-a_0 r).
		\]
		
		%
		\item
		We say that the \textbf{capacity upper bound}
		\hypertarget{cap}{$\operatorname{Cap(\phi)_{\le}}$} holds if there exist
		constants $A_1 \in (1,\infty), A_2 \in [1,\infty), C_1  \in (0,\infty)$ such that for all
		$x \in X$ and $0 < r < \diam(X,d)/A_2$, there exists a cutoff function
		$\psi \in \mathcal{F}$ for $B(x,r) \subset B(x,A_1 r)$ satisfying
		\[
		\mathcal{E}(\psi,\psi)
		\le C_1 \frac{V(x,r)}{\phi(r)}.
		\]
	\end{enumerate}
	
\end{definition}
\begin{remark} \label{r:covering} Let $(X,d,m)$ satisfy the volume doubling property.
	\begin{enumerate}[(i)]
		\item  In the condition \hyperlink{cap}{$\operatorname{Cap}(\phi)_{\le}$}, we may, without loss of generality,
		take $A_1 = 2$. Indeed, this can be achieved by covering $B(x,r)$ with a uniformly bounded
		number of balls of radius $r/A_1$ whose centers lie in $B(x,r)$, and then taking the maximum 
		of the corresponding cutoff functions associated with these smaller balls.  For such a covering argument, see, for instance, \cite[Lemma~6.2]{Mur24}. This argument also allows us to choose $A_2=1$ by slightly increasing $C_1$ if necessary.
		\item
		
		If $(\mathcal E,\mathcal F)$ is regular, the cutoff functions in
		\hyperlink{cap}{$\operatorname{Cap}(\phi)_{\le}$}
		may be chosen continuous. Indeed, let $u\in\mathcal F$
		be a capacity cutoff for $B(x,r)\subset B(x,2r)$.
		Set $K=\overline{B(x,r/2)}$, $U=B(x,3r)$, and
		$\lambda=\phi(r)^{-1}$. Then $K$ is compact,
		$u\in\mathcal F_U$, and
		\[
		\mathcal E(u,u)+\lambda\|u\|_2^2
		\lesssim \frac{V(x,r)}{\phi(r)}.
		\]
		By regularity of the part form
		\cite[Theorem~4.4.3]{FOT} and the $\lambda$-capacity
		version of \cite[Lemma~2.2.7(ii)]{FOT}, followed by
		truncation, there exists $\psi\in\mathcal F\cap C_c(U)$
		with $0\le\psi\le1$, $\psi=1$ on $K$, and
		\[
		\mathcal E(\psi,\psi)
		\lesssim \frac{V(x,r/2)}{\phi(r/2)}.
		\]
		The covering argument in (i) then gives continuous
		cutoff functions with   arbitrary $A_1 \in (1,\infty)$ and $A_2=1$.
		
	\end{enumerate}
	
\end{remark}

By \cite[Proposition 2.3-(5)]{CKW-hke}, under \hyperlink{vd}{$\on{VD}$}, and \hyperlink{jphi}{$\on{J(\phi)_\le}$} we have the implication 
\begin{equation} \label{e:csj-cap}
	\mbox{\hyperlink{csj}{$\operatorname{CSJ(\phi)}$}} \implies \mbox{\hyperlink{cap}{$\operatorname{Cap(\phi)_{\le}}$.}}
\end{equation}

Given a Dirichlet form $(\mathcal{E},\mathcal{F})$, there is an associated \textbf{Markov semigroup} $(P_t)_{t \ge 0}$ on $L^2(X,m)$. Furthermore, by \cite[Theorem 1.3.1 and Lemma 1.3.4]{FOT} the Dirichlet form $(\mathcal{E},\mathcal{F})$ is given in terms of the semigroup by
\begin{align*}
	\mathcal{F}&= \set{f \in L^2(X,m): \lim_{t \downarrow 0} \frac{1}{t}\langle f - P_t f, f \rangle < \infty },  \\
	\mathcal{E}(f,f) &=  \lim_{t \downarrow 0} \frac{1}{t}\langle f - P_t f, f \rangle, \q \mbox{for all $f \in \mathcal{F}$,} 
\end{align*}
where $\langle \cdot,\cdot \rangle$ denotes the $L^2(X,m)$ inner product.
The \textbf{heat kernel} associated with the Markov semigroup $\{P_t\}$ (if it exists) is a family of measurable functions $p(t,\cdot,\cdot):X  \times X  \mapsto [0,\infty)$ for every $t > 0$, such that
\begin{align*}
	P_tf(x) &= \int p(t,x,y)  f(y)\, m(dy), \q \mbox{for all $f \in L^2(X,m), t>0$ and $x \in X$,} \\
	p(t,x,y) &= p(t,y,x), \q \mbox{for all $x, y \in X$ and $t>0$,}   \\
	p(t+s,x,y)&= \int p(s,x,z) p(t,z,y) \,m(dy), \q \mbox{for all $t,s>0$ and $x,y \in X$.}  
\end{align*}

\begin{definition} \label{d:stable}
	We say that $(X,d,m,\mathcal{E},\mathcal{F})$
	satisfies the \textbf{stable-like heat kernel estimates} \hypertarget{shk}{$\on{HK}(\phi)$}
	if there exist $C_1  \in (1,\infty)$ and a heat kernel $\{ p_{t} \}_{t>0}$ of
	the semigroup corresponding to the Dirichlet form $(\mathcal{E},\mathcal{F})$ such that for each $t \in (0,\infty)$,
	\begin{equation} \label{eq:shk-upper}
		p_{t}(x,y) \le C_1 \biggl( \frac{1}{m\bigl(B(x,\phi^{-1}(t))\bigr)} \wedge \frac{t}{m\bigl(B(x,d(x,y))\bigr) \phi(d(x,y))} \biggr)
	\end{equation}
	and
	\begin{equation} \label{eq:shk-lower}
		p_{t}(x,y) \ge C_1^{-1} \biggl( \frac{1}{m\bigl(B(x,\phi^{-1}(t))\bigr)} \wedge \frac{t}{m\bigl(B(x,d(x,y))\bigr) \phi(d(x,y))} \biggr)
	\end{equation}
	for $m$-a.e.\ $x,y \in X$, where $\phi^{-1}(\cdot)$ denotes the
	inverse of the homeomorphism $\phi(\cdot) \colon [0,\infty)\to [0,\infty)$.
\end{definition}

The following is the first main result of the work and provides a sufficient condition for cutoff Sobolev inequality for pure jump process. 
\begin{theorem}  \label{t:main} 
	Let $(X,d,m)$ be a metric measure space  that satisfies 	\hyperlink{vd}{$\on{VD}$}. Let $(\mathcal{E},\mathcal{F})$ be a pure jump Dirichlet form on $L^2(X,m)$ that satisfies 
	\hyperlink{jphi}{$\on{J(\phi)}$}, and \hyperlink{cap}{$\on{Cap(\phi)_{\le}}$}. Then $(\mathcal{E},\mathcal{F})$ satisfies the cutoff energy inequality \hyperlink{csj}{$\operatorname{CSJ}(\phi)$}.
\end{theorem}
%

Theorem \ref{t:main} along with known characterizations of stable-like heat kernel bounds imply  an affirmative answer to \cite[Conjecture 1.14]{GHH18} as we state below.
\begin{theorem} \label{t:char-hke}
	Let $(X,d,m)$ be a metric measure space  that satisfies 	\hyperlink{vd}{$\on{VD}$} and \hyperlink{qrvd}{$\on{QRVD}$}. Let $(\mathcal{E},\mathcal{F})$ be a pure jump Dirichlet form on $L^2(X,m)$. Then the following are equivalent:
	\begin{enumerate}[(a)]
		\item Jump kernel bound  \hyperlink{jphi}{$\on{J(\phi)}$}, and capacity upper bound \hyperlink{cap}{$\on{Cap(\phi)_{\le}}$}.
		\item Stable-like heat kernel bound \hyperlink{shk}{$\on{HK}(\phi)$}.
	\end{enumerate}
\end{theorem}

\begin{remark}
	\begin{enumerate}[(a)]
		\item Theorems \ref{t:main} and \ref{t:char-hke} also hold for the more general notion of regular scale function introduced in \cite[Definition 5.4]{BM}. These extensions follow from the present formulation by a quasi-symmetric change of metric, as in \cite[Proof of Lemma 5.9]{BM}. 
		\item Theorem \ref{t:main} implies that the cutoff Sobolev inequality appearing in the characterization of stable-like heat kernel bounds in the discrete-time graph setting in \cite[Theorem 1.5]{MS} can be replaced by capacity upper bounds analogous to those in Theorem \ref{t:char-hke}.
		\item We expect Theorem \ref{t:char-hke} to simplify the verification of stable-like heat kernel estimates in concrete examples. For instance, the proof of stable-like heat kernel bounds for the boundary trace of reflected diffusions in \cite[Theorem 5.14]{KM} relies on establishing exit time bounds together with jump kernel estimates. In this case, proving capacity upper bounds is substantially simpler than proving exit time bounds; see also Theorem \ref{t:reflected-hk} for another such simplification.
	\end{enumerate}
\end{remark}

\subsection{Outline of the proof}

Since Theorem \ref{t:char-hke} follows readily from Theorem \ref{t:main} together with known results, we only outline the proof of Theorem \ref{t:main}.
The Whitney blending argument proceeds in three main steps:
\begin{enumerate}
	\item We establish a cutoff Sobolev-type inequality in the case where the cutoff function is chosen as an equilibrium potential for $K \subset \Omega$, for functions that vanish either in a neighborhood of $K$ or in a neighborhood of $\Omega^c$ (Lemma \ref{l:zerobdy-cs}).
	
	\item Given $f \in \mathcal{F}$, $x_0 \in X$, and $r>0$, we construct functions $f_1,f_2 \in \mathcal{F}$ such that
	$\lvert f \rvert \le \lvert f_1 \rvert + \lvert f_2 \rvert$, where $f_1$ vanishes in a neighborhood of $\overline{B(x_0,r)}$ and $f_2$ vanishes in a neighborhood of $B(x_0,2r)^c$.
	This allows us to apply the cutoff Sobolev-type inequality from  the previous step to $f_1$ and $f_2$, yielding a cutoff Sobolev inequality for an equilibrium potential associated with $\overline{B(x_0,r)} \subset B(x_0,2r)$.
	A key difficulty is controlling the energies of $f_1$ and $f_2$ in terms of the energy of $f$.
	This is accomplished using Whitney covers and constitutes the most delicate part of the proof (see Lemma \ref{l:wb}).
	The construction of $f_1,f_2$ also relies on the explicit description of the domain of the Dirichlet form established in Proposition \ref{p:domain}.

	\item The previous step produces a non-local analogue of the simplified cutoff Sobolev inequality introduced in \cite[Definition 6.1]{Mur24} (see Proposition \ref{p:css}).
	A self-improvement argument, following \cite[Proof of Lemma~6.2]{Mur24}, then yields the desired cutoff Sobolev inequality (see \textsection \ref{ss:self-improve}).
\end{enumerate}
{\bf Notation}. In the following, we will use the notation  $A \lesssim B$ for quantities $A$ and $B$ to indicate the existence of an
implicit constant $C \ge 1$ depending on some inessential parameters such that $A \le CB$. We write $A \asymp B$, if $A \lesssim B$ and $B \lesssim A$.
\section{Proofs}

\subsection{Identification of the domain of the Dirichlet form}
The Whitney blending technique involves the construction of a piecewise defined function
(see \eqref{e:def-ext}). In order to verify that such a function belongs to the domain of the
Dirichlet form, it is necessary to have a concrete sufficient condition ensuring membership
in the domain. In the local setting, such a criterion is provided in \cite[Lemma~4.9]{Eri}.
Here, we develop a non-local analogue based on an explicit description of the domain.

Let $(\mathcal{E},\mathcal{F})$ be a pure jump Dirichlet form  on $L^2(X,m)$  with jump measure $J$. A natural candidate for the domain $\mathcal{F}$   is the following subspace:
\begin{equation} \label{e:def-domain}
	\widehat{\mathcal{F}}
	:= \left\{ f \in L^2(X,m) : \int_{X \times X} (f(x)-f(y))^2 \, J(dx,dy) < \infty \right\}.
\end{equation}
While it is immediate that \(\mathcal{F} \subset \widehat{\mathcal{F}}\),
the converse inclusion requires proof. Several ideas used in establishing the equality
\(\mathcal{F} = \widehat{\mathcal{F}}\) in following proposition will later reappear in the Whitney blending argument
in \textsection~\ref{ss:whitney}. We also note that a considerably shorter proof of this result
was obtained earlier under the \emph{a priori} stronger assumption of stable-like heat kernel
estimates in \cite[Theorem~2.40(d)]{KM}.

\begin{prop} \label{p:domain} 	Let $(X,d,m)$ be a metric measure space  that satisfies 	\hyperlink{vd}{$\on{VD}$}. Let $(\mathcal{E},\mathcal{F})$ be a pure jump Dirichlet form that satisfies 
	\hyperlink{jphi}{$\on{J(\phi)}$}, and \hyperlink{cap}{$\on{Cap(\phi)_{\le}}$}. Then we have $\mathcal{F} = \wh{\mathcal{F}}$.
\end{prop}
We first establish some preliminary lemmas that are needed for the proof of Proposition \ref{p:domain}.
\begin{lem} \label{l:ujs} If the jump measure $J$  of the pure-jump Dirichlet form $(\mathcal{E},\mathcal{F})$ on $L^2(X,m)$ satisfies \hyperlink{jphi}{$\operatorname{J(\phi)}$}, then there exists $C>0$ such that 
	for all $x,y \in X$ and $r_1,r_2 \in (0,d(x,y)/4]$, we have 
	\[
	\frac{1}{V(x,r_1) V(y,r_2)}	\int_{B(x,r_1)} \int_{B(y,r_2)} J(z,w)  \,m(dz)\,m(dw)\le  C J(x,y),
	\]
	and 
	\[
	\frac{1}{V(y,r_2)} \int_{B(y,r_2)} J(x,w)\,m(dw) \le C J(x,y).
	\]
\end{lem}
\begin{proof}
	Let $x,y \in X$,  $r_1,r_2 \in (0,d(x,y)/4]$.    Then by the triangle inequality, 
	\[
	d(x,y)/2 \le d(x,y)-r_1-r_2  \le d(z,w) \le d(x,y)+r_1+r_2 \le 2 d(x,y)
	\] 
	for all $z \in B(x,r_1), w \in B(y,r_2)$. Hence for all $z \in B(x,r_1), w \in B(y,r_2)$,   we have 
	\[
	J(z,w) \asymp \frac{1}{V(z,d(z,w)) \phi(d(z,w))} \stackrel{\eqref{e:vd-iter}, \eqref{e:scale}}{\asymp} \frac{1}{V(x,d(x,y))\phi(d(x,y))} \asymp J(x,y).
	\]
\end{proof}

We recall that an \emph{$r$-net} $N$ of  $X$ is a maximal $r$-separated subset of $X$; that is $d(x,y) \in \{0\} \cup [r,\infty)$ for all $x,y \in N$ and for any $\wt{N}\supsetneqq N$, there exists $x,y \in  \wt{N}$ such that $d(x,y) \in (0,r)$. We record the following \emph{bounded overlap property} for balls centered at an $r$-net.
For any $A>0$, there exists a constant $C=C(A)$, depending only on $A$ and the constants
in the volume doubling property \hyperlink{vd}{$\on{VD}$}, such that for any $r>0$ and any
$r$-net $N$,
\begin{equation} \label{e:overlap}
	\sum_{x \in N} \one_{B(x,Ar)} \le C\,\one .
\end{equation}
The following construction of partition of unity indexed by a net is standard in the setting of strongly local Dirichlet forms.
However, the non-local case requires a slightly different argument, and the verification
that $\psi_x \in \mathcal{F}$ is often omitted in the literature (see, for example,
\cite[Lemma 2.5]{Mur20} and the references therein). For this reason, we provide a detailed proof.

\begin{lem} \label{l:partition} Let $(X,d,m)$ be a metric measure space  that satisfies 	\hyperlink{vd}{$\on{VD}$}. Let $(\mathcal{E},\mathcal{F})$ be a pure jump Dirichlet form that satisfies 
	\hyperlink{cap}{$\on{Cap(\phi)_{\le}}$}.
	There exist $C,c>0$ such that for any $r>0$ and for any $r$-net $N$, there exists a collection of functions $\{\psi_x: x \in N\}$ satisfying the following properties:
	\begin{enumerate}[(i)]
		\item $\sum_{x \in N} \psi_x \equiv \one$.
		\item For each $x \in X$, $\psi_x \in \mathcal{F} \cap C(X)$, 
		\begin{equation} \label{e:pu}
			c \one_{B(x,r)}\le\psi_x \le \one_{B(x,2r)}, \quad \mbox{ and}, \quad \mathcal{E}(\psi_x,\psi_x) \le C \frac{V(x,r)}{\phi(r)}.
		\end{equation}
	\end{enumerate} 
\end{lem}

\begin{proof} 
	Let $N$ be any $r$-net, where $r>0$.
	By   \hyperlink{cap}{$\operatorname{Cap(\phi)_{\le}}$},  Remark \ref{r:covering}, and the regularity of the Dirichlet form,  for each $x \in N$,   there exists
	$\wt{\psi}_x \in \mathcal{F} \cap C(X)$ be such that 
	\[
	\wt{\psi}_x \in \mathcal{F},  \quad \one_{B(x,r)} \le \wt{\psi}_x \le \one_{B(x,2r)}, \quad \mathcal{E}(\wt{\psi}_x,\wt{\psi}_x) \lesssim \frac{V(x,r)}{\phi(r)}.
	\]
	We claim that the family  $\{\psi_x : x \in N\}$, defined by 
	\begin{equation} \label{e:pu0}
		\psi_x:= \frac{\wt{\psi}_{x}}{\sum_{y \in N} \wt{\psi}_y}, \mbox{for all $x \in N$,}
	\end{equation}
	satisfies the desired properties. By the maximality of the net and the bounded overlap property \eqref{e:overlap}, there exists $c \in (0,1)$ such that 
	\begin{equation} \label{e:pu1}
		\one \le \sum_{x \in N} \wt{\psi}_x \le c^{-1} \one.
	\end{equation}
	Using  McSchane's extension theorem \cite[Theorem 6.2]{Hei}, we pick a $\sqrt{2}$-Lipschitz function $T:\mathbb{R}^2 \to \mathbb{R}$ such that 
	\[
	T(a,b)= \begin{cases}
		a/b & \mbox{if $a \in [0,1], b \in [1,\infty)$,}\\
		0 & \mbox{if $a=0,b \ge 0$.}
	\end{cases}
	\]
	Hence, combining \eqref{e:pu0} and \eqref{e:pu1} with the generalized contraction property of Dirichlet forms (\cite[Theorem A.2]{KS} or \cite[Exercise I.3.5, Proof of Proposition I.3.3.1]{BH91}), we obtain that
	\[
	\psi_x=  T\left(\wt{\psi}_x, \sum_{y \in N \cap B(x,4r)} \wt{\psi}_y \right) \in \mathcal{F},  
	\]
	and 
	\begin{align*}
		\mathcal{E}(\psi_x,\psi_x) &\le 2 \left(\mathcal{E}(\wt{\psi}_x,\wt{\psi}_x) + \mathcal{E} \left(\sum_{y \in N \cap B(x,4r)} \wt{\psi}_y ,\sum_{y \in N \cap B(x,4r)} \wt{\psi}_y  \right)  \right)  \\
		& \le 2 \left(\mathcal{E}(\wt{\psi}_x,\wt{\psi}_x) + \#(N \cap B(x,4r)) \sum_{y \in N \cap B(x,4r} \mathcal{E}(\wt{\psi}_y,\wt{\psi}_y) ) \right) \\
		&\stackrel{\eqref{e:overlap}, \eqref{e:pu}}{\lesssim} \sum_{y \in N \cap B(x,4r} \frac{V(y,r)}{\phi(r)} \stackrel{\eqref{e:vd-iter}}{\lesssim}  \frac{V(x,r)}{\phi(r)}.
	\end{align*}
\end{proof}

\begin{proof}[Proof of Proposition \ref{p:domain}.]
	Let $x_0 \in X, r>0$ and $f \in \wh{\mathcal{F}}$. Let $N$ be any $r$-net, where $r>0$. Let $\{\psi_x: x \in N\}$ denote the partition of unity satisfying the conclusion of Lemma \ref{l:partition}. We define for each $n \in \mathbb{N}$, 
	\[
	f_{r,n}(\cdot):= \sum_{y \in B(x_0,n) \cap N} f_{B(y,r)} \psi_y(\cdot),  
	\]
	and set $f_r(\cdot):= \sum_{y \in N} f_{B(y,r)} \psi_y(\cdot), F_r(\cdot):= \sum_{y \in N} \abs{f_{B(y,r)}} \psi_y(\cdot)$.
	
	First, we note that by the bounded overlap property \eqref{e:overlap} and Lemma \ref{l:partition}, we have
	\begin{align} \label{e:dom1}
		\int \abs{F_r}^2 \,dm  &\stackrel{\eqref{e:overlap}}{\lesssim}  \int \sum_{y \in N} \abs{f_{B(y,r)}}^2 \psi_y^2 \,dm    \stackrel{\eqref{e:pu}}{\lesssim} \sum_{y \in N}\abs{f_{B(y,r)}}^2 V(y,2r)  \stackrel{\on{VD}}{\lesssim} \sum_{y \in N}\abs{f_{B(y,r)}}^2 V(y,r) \nonumber \\
		&\lesssim \sum_{y \in N} \int_{B(y,r)} f^2\,dm  \stackrel{\eqref{e:overlap}}{\lesssim} \int f^2\,dm <\infty. \quad \mbox{(by Jensen's inequality)}
	\end{align}
	Therefore by the dominated convergence theorem $f_{r,n} \to f_r$ in $L^2(X,m)$. Since $f_{r,n}$ is a finite linear combination of functions in $\mathcal{F}$, we have that $f_{r,n} \in \mathcal{F}$ for all $f \in \wh{\mathcal{F}}$ and $n \in \mathbb{N}$. 
	
	Next, let us estimate the energy $\mathcal{E}(f_{r,n},f_{r,n})$.
	By the bounded overlap property \eqref{e:overlap}, and Cauchy-Schwarz inequality, we have  for all $x,y \in X, r >0, f \in \wh{\mathcal{F}}$ and $n \in \mathbb{N}$,
	\begin{align} \label{e:dom1b}
		(f_{r,n}(x)- f_{r,n}(y))^2 &= \left( \sum_{w \in N \cap B(x_0,n)} f_{B(w,r)} (\psi_w(x)-\psi_w(y)) \right)^2 \nonumber \\
		&\lesssim \sum_{w \in N \cap B(x_0,n)} f_{B(w,r)}^2 (\psi_w(x)-\psi_w(y))^2.
	\end{align}
	Therefore for all $n \in \mathbb{N}, r>0$ and $f \in \wh{\mathcal{F}}$, we have
	\begin{align*}
		\mathcal{E}(f_{r,n},f_{r,n}) &=  \int (f_{r,n}(x)- f_{r,n}(y))^2\, J(dx,dy)\\
		&\le \sum_{w \in N \cap B(x_0,n)} f_{B(w,r)}^2  \int (\psi_w(x)-\psi_w(y))^2 \,J(dx,dy) \le \sum_{w \in N} f_{B(w,r)}^2 \mathcal{E}(\psi_w,\psi_w)\\
		&\stackrel{\eqref{e:pu}}{\lesssim} \sum_{w \in N} f_{B(w,r)}^2 \frac{V(w,r)}{\phi(r)} \stackrel{\eqref{e:dom1}}{\lesssim} \frac{1}{\phi(r)} \int_X f^2\,dm.
	\end{align*}
	In particular,  $\liminf_{n \to \infty } 	\mathcal{E}(f_{r,n},f_{r,n})< \infty$. Combining this with $f_{r,n} \to f_r$ in $L^2(X,m)$ and using the fact that $(\mathcal{E},\mathcal{F})$ is a closed form, we conclude that 
	$f_r \in \mathcal{F}$ and 
	\begin{equation} \label{e:dom1a}
		\mathcal{E}(f_r,f_r) \le \liminf_{n \to \infty } 	\mathcal{E}(f_{r,n},f_{r,n}) \lesssim \frac{1}{\phi(r)} \int_X f^2\,dm < \infty.
	\end{equation}
	Here, the conclusions $f_{r} \in \mathcal{F}$ and \eqref{e:dom1a} follow from the standard equivalence between closedness of the quadratic form $(\mathcal{E},\mathcal{F})$ on $L^2(X,m)$ and lower semicontinuity on $L^2(X,m)$ of the functional
	\[
	g \mapsto
	\begin{cases}
		\mathcal{E}(g,g), & \text{if } g \in \mathcal{F},\\
		\infty, & \text{otherwise},
	\end{cases}
	\]
	see \cite[p.~372, \textsection 1(c)]{Mos}.

	Similar to the argument above showing that $f_r \in \mathcal{F}$, to prove that
	$f \in \mathcal{F}$ it suffices to show that
	\begin{equation} \label{e:dom2}
		\lim_{r \downarrow 0} \int_X \abs{f - f_r}^2 \, dm = 0 . 
	\end{equation}
	
	and 
	\begin{equation} \label{e:dom3}
		\liminf_{r \downarrow 0} \mathcal{E}(f_r,f_r) < \infty.
	\end{equation}

	Next, we show \eqref{e:dom2}.
	To this end, note that for any $r>0$,
	\begin{align} \label{e:dom4}
		\int \abs{f-f_r}^2\,dm&= \int \abs{f(x)-\sum_{y \in N} f_{B(y,r)} \psi_y(x)}^2\, m(dx)=  \int \abs{\sum_{y \in N} (f(x)-f_{B(y,r)}) \psi_y(x)}^2\, m(dx) \nonumber \\
		&{\lesssim} \sum_{y \in N} \int  \abs{(f(x)-f_{B(y,r)}) \psi_y(x)}^2\, m(dx) \quad \mbox{(by \eqref{e:overlap}, \eqref{e:pu}, Cauchy--Schwarz)} \nonumber \\
		&\lesssim \sum_{y \in N} \int_{B(y,2r)}  \abs{(f(x)-f_{B(y,r)})}^2\, m(dx) ,\quad \mbox{(since $\psi_y \le \one_{B(y,2r)}$)} \nonumber \\
		&\lesssim \sum_{y \in N} \frac{1}{V(y,r)} \int\int  \one_{B(y,2r)}(x_1)\one_{B(y,r)}(x_2) \abs{(f(x_1)-f(x_2))}^2\, m(dx_1)\, m(dx_2)\nonumber \\
		&\quad\quad \mbox{(by Jensen's inequality)}\nonumber \\
		&\lesssim \sum_{y \in N}   \int\int  \one_{B(y,2r)}(x_1)\one_{B(y,r)}(x_2) \frac{\abs{(f(x_1)-f(x_2))}^2}{V(x_1,r)}\, m(dx_1)\, m(dx_2) \quad \mbox{(by \eqref{e:vd-iter})} 
	\end{align}
	By the bounded overlap property \eqref{e:overlap}, we have 
	\begin{equation} \label{e:dom5}
		\sum_{y \in N} \one_{B(y,2r)}(x_1)\one_{B(y,r)}(x_2) \lesssim \one_{B(x_2,3r)}(x_1), \quad \mbox{for all $x_1,x_2 \in X$.}
	\end{equation}
	If $d(x_1,x_2) < 3r$, then by \eqref{e:scale}, \hyperlink{jphi}{$\on{J(\phi)}$} and \eqref{e:vd-iter}, we have 
	\begin{equation} \label{e:dom6}
		\frac{1}{V(x_1,r)} \lesssim    \frac{1}{V(x_1,d(x_1,x_2))} \lesssim \phi(r) \frac{1}{V(x_1,d(x_1,x_2))\phi(d(x_1,x_2))} \lesssim \phi(r) J(x_1,x_2).
	\end{equation}
	Combining \eqref{e:dom4}, \eqref{e:dom5}, and \eqref{e:dom6}, we have 
	\begin{align} \label{e:dom7}
		\int \abs{f-f_r}^2\,dm  &\stackrel{\eqref{e:dom4}, \eqref{e:dom5}}{\lesssim}  \int\int   \one_{B(x_1,3r)}(x_2) \frac{\abs{(f(x_1)-f(x_2))}^2}{V(x_1,r)}\, m(dx_1)\, m(dx_2) \nonumber\\
		&\stackrel{\eqref{e:dom6}}{\lesssim} \phi(r) \int\int  \abs{(f(x_1)-f(x_2))}^2 J(x_1,x_2)\, m(dx_1)\,m(dx_2)
	\end{align}
	By $f \in \wh{\mathcal{F}}$, \eqref{e:dom7}, and \eqref{e:scale}, we obtain \eqref{e:dom2}.

	Since the bound \eqref{e:dom1a} is not enough to obtain \eqref{e:dom3}, we need an improved bound on $$ \mathcal{E}(f_r,f_r)= \int (f_r(x)-f_r(y))^2 \,J(dx,dy).$$  
	We estimate the above integral by considering two cases, according to whether $d(x,y)<18r$. To this end, we need two estimates on $(f_r(x)-f_r(y))^2$. The first estimate is obtained by the same argument as in \eqref{e:dom1b},
	yielding the following estimate for all $x,y \in X$, $r>0$,
	$f \in \widetilde{\mathcal{F}}$, and $K \in \mathbb{R}$:
	\begin{align} \label{e:dom8}
		(f_r(x)-f_r(y))^2 &= \left(\sum_{w \in N} (f_{B(w,r)}-K) (\psi_w(x)-\psi_w(y))\right)^2 \nonumber \\
		&\lesssim  \sum_{w \in N} (f_{B(w,r)}-K)^2 (\psi_w(x)-\psi_w(y))^2.
	\end{align}
	The second estimate for $(f_r(x) - f_r(y))^2$ follows from the Cauchy--Schwarz
	inequality and \eqref{e:pu}, yielding
	\begin{align} \label{e:dom9}
		(f_r(x)-f_r(y))^2 &= \left( \sum_{w_1,w_2 \in N}\left(f_{B(w_1,r)}-  f_{B(w_2,r)}\right)\psi_{w_1}(x) \psi_{w_2}(y)\right)^2 \nonumber \\
		&\le \#(N \cap B(x,2r)) \cdot \#(N \cap B(y,2r)) \sum_{\substack{w_1 \in N \cap B(x,2r),  \\
				w_2 \in N \cap B(y,2r)}}  \left(f_{B(w_1,r)}-  f_{B(w_2,r)}\right)^2 \nonumber \\
		&\stackrel{\eqref{e:overlap}}{\lesssim } \sum_{\substack{w_1 \in N \cap B(x,2r), \\
				w_2 \in N \cap B(y,2r)}}  \left(f_{B(w_1,r)}-  f_{B(w_2,r)}\right)^2 \nonumber \\
		&\lesssim\int_{B(x,3r)} \int_{B(y,3r)}  \frac{(f(z_1)-f(z_2))^2 }{V(z_1,r)V(z_2,r)}  \,m(dz_1) \,m(dz_2),
	\end{align}
	where the last line follows from by \eqref{e:vd-iter}, Jensen's inequality and $B(w_1,r) \subset B(x,3r), B(w_2,r) \subset B(y,3r)$ for  all $w_1 \in N \cap B(x,2r),  
	w_2 \in N \cap B(y,2r)$.
	Note that, if $d(x,y) \ge 18r$, then for any $z_1 \in B(x,3r), z_2 \in B(y,3r)$, we have 
	\[
	\frac{1}{4}  d(z_1,z_2) \ge \frac{1}{4} (d(x,y) - 6r) \ge 3r, 
	\]
	and hence by Lemma \ref{l:ujs} and \eqref{e:vd-iter}, we have 
	\begin{equation} \label{e:dom10}
		\frac{1}{V(z_1,r) V(z_2,r)}	\int_{B(z_1,3r)} \int_{B(z_2,3r)} J(w_1,w_2)  \,m(dw_1)\,m(dw_2)\lesssim  C J(z_1,z_2).
	\end{equation}
	Using the estimate \eqref{e:dom9} and interchanging the order of integration
	via Fubini's theorem, we obtain
	\begin{align}\label{e:dom11}
		\MoveEqLeft{	\int \int (f_r(x)-f_r(y))^2 \one_{\{d(x,y) \ge 18r\}} J(x,y) \,m(dx)\,m(dy)} \nonumber \\
		&\stackrel{\eqref{e:dom9}}{\lesssim} \int \int  \int_{B(z_2,3r)} \int_{B(z_1,3r)} \one_{\{d(x,y) \ge 18r\}}  \frac{(f(z_1)-f(z_2))^2 }{V(z_1,r)V(z_2,r)} J(x,y)  \,m(dx)\,m(dy)\,m(dz_1)\,m(dz_2) \nonumber \\
		&\stackrel{\eqref{e:dom10}}{\lesssim} \int \int (f(z_1)-f(z_2))^2 J(z_1,z_2)\, m(dz_1)\,m(dz_2).
	\end{align}
	To estimate the integral under $d(x,y) < 18r$, we use that 
	\[
	\one_{B(x,18r)}(y) \le \sum_{w \in N} \one_{B(w,r)}(x)\one_{B(w,19r)}(y) \le \sum_{w \in N} \one_{B(w,19r)}(x)\one_{B(w,19r)}(y), 
	\]
	for all $x,y \in X$ such that $d(x,y) < 18r$. Furthermore, for such $x,y \in X$, using \eqref{e:pu}, we observe that
	\[
	\psi_w(x)-\psi_w(y)=0, \quad \mbox{for all $w \in  N\cap B(x,20r)^c$.}
	\]
	Combining the two displays above, and choosing $K = f_{B(x,r)}$
	in \eqref{e:dom8}, we obtain that for all $x,y \in X$ with
	$d(x,y) < 18r$,
	\begin{align} \label{e:dom12}
		\MoveEqLeft{(f_r(x)-f_r(y))^2}  &\lesssim \sum_{w \in N \cap B(x,20r) \cap B(y,20r)}  (f_{B(w,r)}-f_{B(x,r)})^2 (\psi_w(x)-\psi_w(y))^2.
	\end{align}
	By Jensen's inequality and \eqref{e:vd-iter}, for any $w \in N$
	and any $x \in B(w,20r)$, we have
	\begin{align} \label{e:dom13}
		\MoveEqLeft{\abs{f_{B(w,r)}-f_{B(x,r)}}^2 } \nonumber\\&\lesssim \frac{1}{m(B(w,r))^2} \int_{B(w,21 r)} \int_{B(w,21 r)} (f(z_1)-f(z_2))^2 \,m(dz_1)\,m(dz_2) \nonumber \\
		&\stackrel{\eqref{e:scale}}{\lesssim} \frac{\phi(r)}{m(B(w,r))} \int_{B(w,21 r)} \int_{B(w,21 r)} \frac{(f(z_1)-f(z_2))^2}{V(z_1,d(z_1,z_2)) \phi(d(z_1,z_2))} \,m(dz_1)\,m(dz_2) \nonumber \\
		&\lesssim \frac{\phi(r)}{m(B(w,r))} \int_{B(w,21 r)} \int_{B(w,21 r)} (f(z_1)-f(z_2))^2 J(z_1,z_2) \,m(dz_1)\,m(dz_2) 
	\end{align}
	Hence by \eqref{e:dom12}, \eqref{e:dom13}, we obtain
	\begin{align} \label{e:dom14}
		\MoveEqLeft{	\int \int (f_r(x)-f_r(y))^2 \one_{\{d(x,y) < 18r\}} J(x,y) \,m(dx)\,m(dy)} \nonumber \\
		&\lesssim \sum_{w \in N} \frac{\phi(r)}{m(B(w,r))} \int_{B(w,21 r)} \int_{B(w,21 r)} (f(z_1)-f(z_2))^2 J(z_1,z_2) \,m(dz_1)\,m(dz_2)  \mathcal{E}(\psi_w,\psi_w) \nonumber \\
		&\stackrel{\eqref{e:pu}}{\lesssim} \sum_{w \in N}  \int_{B(w,21 r)} \int_{B(w,21 r)} (f(z_1)-f(z_2))^2 J(z_1,z_2) \,m(dz_1)\,m(dz_2) \nonumber \\
		&\lesssim  \int_{X} \int_{X} (f(z_1)-f(z_2))^2 J(z_1,z_2) \,m(dz_1)\,m(dz_2) \quad \mbox{(by \eqref{e:overlap}).}
	\end{align}
	Combining \eqref{e:dom11} and \eqref{e:dom14}, we obtain \eqref{e:dom3} which concludes the proof that $f \in \mathcal{F}$ and hence $\mathcal{F}= \wh{\mathcal{F}}$.
\end{proof}
\subsection{Truncated Dirichlet form}
A useful and well-known technique for handling difficulties arising from non-local
Dirichlet forms is to truncate the jump kernel so that it behaves more like a local
Dirichlet form. For a probabilistic interpretation of the resulting decomposition
of the jump kernel in terms of the Meyer decomposition, we refer the reader to
\cite[\textsection~3]{BGK}.

Let \((\mathcal{E},\mathcal{F})\) be a pure jump Dirichlet form with jump measure
\(J(dx,dy)=J(x,y)\,m(dx)\,m(dy)\). For \(\rho>0\), we define the \(\rho\)-truncated
Dirichlet form \((\mathcal{E}^{(\rho)},\mathcal{F})\) by
\begin{equation} \label{e:def-trun-df}
	\mathcal{E}^{(\rho)}(f,g)
	:= \int (f(x)-f(y))(g(x)-g(y))\, J^{(\rho)}(dx,dy),
	\qquad f,g \in \mathcal{F},
\end{equation}
where
\begin{equation} \label{e:def-truncate}
	J^{(\rho)}(x,y)
	:= \one_{\{d(x,y)\le \rho\}} J(x,y),
	\qquad
	J^{(\rho)}(dx,dy)
	:= \one_{\{d(x,y)\le \rho\}} J(dx,dy).
\end{equation}

By \cite[Lemma~2.1 and (2.3)]{CKW-hke} and \hyperlink{jphi}{$\mathrm{J}(\phi)_{\le}$},
there exists a constant \(C>0\) such that for all \(\rho>0\) and all \(f \in \mathcal{F}\),
\begin{equation} \label{e:trunc-est-df}
	0 \le \mathcal{E}(f,f)-\mathcal{E}^{(\rho)}(f,f)
	\le \frac{C}{\phi(\rho)}\int_X f^2\,dm.
\end{equation}

Denote the corresponding energy measure by
\(\Gamma^{(\rho)}(f,f)(dx)=\Gamma^{(\rho)}(f,f)(x)\,m(dx)\), where
\[
\Gamma^{(\rho)}(f,f)(x)
:= \int_X (f(y)-f(x))^2\, J^{(\rho)}(x,y)\,m(dy).
\]
By \cite[Lemma~2.1]{CKW-hke}, there exists \(C>0\) such that for any
\(f \in \mathcal{F} \cap L^\infty(X,m)\),
\begin{equation} \label{e:trunc-est-em}
	0 \le \Gamma(f,f)(x)-\Gamma^{(\rho)}(f,f)(x)
	\le \frac{C}{\phi(\rho)} \norm{f}_\infty^2,
	\quad \text{for } m\text{-a.e. } x \in X.
\end{equation}

Similarly, by \cite[Lemma~2.1]{CKW-hke}, there exists \(C>0\) such that for any
\(x \in X\) and \(r>0\),
\begin{equation} \label{e:jump-tail}
	\int_{B(x,r)^c} J(x,y)\,m(dy) \le \frac{C}{\phi(r)}.
\end{equation}

\subsection{Estimates on an equilibrium potential}

In this subsection, we establish a cutoff Sobolev-type inequality for functions with zero boundary values; see Lemma \ref{l:zerobdy-cs}. The cutoff function is chosen to be an equilibrium potential.

Let $(\mathcal{E},\mathcal{F})$ be a regular Dirichlet form on $L^2(X,m)$, where $(X,d,m)$ is a metric measure space.  The \textbf{$1$-capacity} of a set $A \subset X$ is defined by
\begin{equation} \label{e:defCap1}
	\Cap_1(A) := \inf \bigl\{ \mathcal{E}_{1}(f,f) \bigm\vert \textrm{$f \in \mathcal{F}$, $f \geq 1$ $m$-a.e.\ on a neighborhood of $A$} \bigr\},
\end{equation}
where $\mathcal{E}_{1}:=\mathcal{E}+\langle\cdot,\cdot\rangle_{L^{2}(X,m)}$ as defined before.
A subset $N$ of $X$ is said to be \emph{$\mathcal{E}$-polar} if $\Cap_{1}(N)=0$. If $\Cap_1(A)>0$, then $A \subset X$ is said to be non-$\mathcal{E}$-polar or \emph{non-polar}.
For $A \subset X$ and a statement $\mathcal{S}(x)$ on $x \in A$, we say that
$\mathcal{S}$ holds \emph{$\mathcal{E}$-quasi-everywhere on $A$} (\emph{$\mathcal{E}$-q.e.\ on $A$} for short),
or $\mathcal{S}(x)$ holds for \emph{$\mathcal{E}$-quasi-every} $x \in A$ (\emph{$\mathcal{E}$-q.e.}\ $x \in A$ for short),
if $\mathcal{S}(x)$ holds for any $x \in A \setminus N$ for some $\mathcal{E}$-polar $N \subset X$.
When $A=X$, we often write just ``\emph{$\mathcal{E}$-q.e.}''\ instead of ``$\mathcal{E}$-q.e.\ on $X$''.
A non-decreasing sequence $\{F_{k}\}_{k\in\mathbb{N}}$ of closed subsets of $X$
is called an \emph{$\mathcal{E}$-nest} if $\lim_{k\to\infty}\Capa_1(K \setminus F_{k})=0$
for any compact subset $K$ of $X$.
A function $f \colon D \setminus N \to [-\infty,\infty]$, defined $\mathcal{E}$-q.e.\ on
an open subset $D$ of $X$ for some $\mathcal{E}$-polar $N \subset X$,
is said to be \emph{$\mathcal{E}$-quasi-continuous on $D$} if there exists an $\mathcal{E}$-nest
$\{F_{k}\}_{k\in\mathbb{N}}$ such that $F_{k} \cap N = \emptyset$ and $f|_{D \cap F_{k}}$
is an $\mathbb{R}$-valued continuous function on $D \cap F_{k}$ for any $k \in \mathbb{N}$
(again, when $D=X$, we often omit ``on $X$''). For each $f \in \mathcal{F}_e$,
an $\mathcal{E}$-quasi-continuous $m$-version $\widetilde{f}$ of $f$ exists by \cite[Theorem 2.1.7]{FOT}
and is unique $\mathcal{E}$-q.e.\ by \cite[Lemma 2.1.4]{FOT}.

Let $A\Subset\Omega\subset X$, where $\Omega$ is a
precompact open set. Define
\begin{equation}\label{e:defCapD-alt}
	\Cap(A,\Omega^c)
	:=
	\inf\bigl\{
	\mathcal E(f,f):
	f\in\mathcal F,\ 
	\widetilde f\ge1\ \mathcal E\text{-q.e.\ on }A,\ 
	\widetilde f=0\ \mathcal E\text{-q.e.\ on }\Omega^c
	\bigr\}.
\end{equation}

Regularity ensures that the admissible class is nonempty.
By truncation, a minimizing sequence may be chosen
to take values in $[0,1]$. Since $m(\Omega)<\infty$,
this sequence is bounded in $(\mathcal F,\mathcal E_1)$.
The constraints define a closed convex subset of this
Hilbert space and are therefore preserved under weak
convergence. Weak lower semicontinuity of $\mathcal E$
then yields a minimizer $e_{A,\Omega}\in\mathcal F$
with $0\le e_{A,\Omega}\le1$ and

\begin{equation}\label{e:eq-pot}
	\widetilde e_{A,\Omega}=1
	\quad\mathcal E\text{-q.e.\ on }A,
	\qquad
	\widetilde e_{A,\Omega}=0
	\quad\mathcal E\text{-q.e.\ on }\Omega^c.
\end{equation}

We fix such a minimizer and call it an
\textbf{equilibrium potential} of $A$ in $\Omega$.
No uniqueness is asserted.

The following Lemma is a non-local analogue of \cite[Lemma 3.1]{Eri} and provides estimates of energy measure of equilibrium potential. The proof of the analogous lemma in \cite{Eri} relies on strong locality. 
\begin{lemma} \label{l:logcacc}
	Let $(\mathcal{E},\mathcal{F})$ be a regular, pure jump Dirichlet form on $L^2(X,m)$ with jump measure $J(dx,dy)= J(x,y)\,m(dx)\,m(dy)$.
	Let $\psi$ be an equilibrium potential for $K \subset \Omega$, where $\Omega$ is a non-polar, precompact open subset and $K$ is compact. Let $g \in \mathcal{F} \cap C(X)$ be such that $0 \le g \le 1$ and $g \equiv 1$ for some $A \subset X$. Suppose one of the following conditions holds:
	\begin{enumerate}[(a)]
		\item $g \equiv 0$ on some neighborhood of $\Omega^c$, or
		\item $g \equiv 0$ on some neighborhood of $K$.
	\end{enumerate}
	Then $\Gamma(\psi,\psi)(A) \le \mathcal{E}(g,g)$.
\end{lemma}

\begin{proof}
	Let us first assume condition (a). By a quasi-continuous modification if necessary, we assume that $\psi$ is quasi-continuous. Define $u:= \psi \vee g$. Then by \cite[Theorem 1.4.2-(i)]{FOT} and \eqref{e:eq-pot}, $u \in \mathcal{F}$, $u \equiv 0$~q.e.~on $\Omega^c$ and $u \equiv 1$~q.e.~on $K$. Hence by the energy minimizing property of $\psi$, we have
	\begin{equation} \label{e:lcac1}
		\int_{X \times X} (u(x)-u(y))^2\,J(dx,dy)
		= \mathcal{E}(u,u)\ge \mathcal{E}(\psi,\psi)
		= \int_{X \times X} (\psi(x)-\psi(y))^2\, J(dx,dy).
	\end{equation}
	Note that
	\begin{align} \label{e:lcac2}
		(u(x)-u(y))^2 &=0,
		& \quad \mbox{for all $x,y \in A$;} \nonumber \\
		(u(x)-u(y))^2 &\le (g(x)-g(y))^2
		& \quad \mbox{for all $(x,y)\in(A\times A^c)\cup(A^c\times A)$;} \nonumber \\
		(u(x)-u(y))^2 &\le (\psi(x)-\psi(y))^2+(g(x)-g(y))^2
		& \quad \mbox{for all $x,y\in A^c$.}
	\end{align}
	Combining \eqref{e:lcac1} and \eqref{e:lcac2}, and using the symmetry of $J$, we have
	\begin{align*}
		\MoveEqLeft{\int_{X \times X}(\psi(x)-\psi(y))^2\,J(dx,dy)}\\
		&\le 2\int_{A\times A^c}(g(x)-g(y))^2\,J(dx,dy)
		+\int_{A^c\times A^c}
		\left((\psi(x)-\psi(y))^2+(g(x)-g(y))^2\right)\,J(dx,dy),
	\end{align*}
	which can be rearranged to
	\begin{align*}
		\MoveEqLeft{
			2\int_{A\times A^c}(\psi(x)-\psi(y))^2\,J(dx,dy)
			+\int_{A\times A}(\psi(x)-\psi(y))^2\,J(dx,dy)}\\
		&\le 2\int_{A\times A^c}(g(x)-g(y))^2\,J(dx,dy)
		+\int_{A^c\times A^c}(g(x)-g(y))^2\,J(dx,dy).
	\end{align*}
	Therefore, since $g\equiv1$ on $A$,
	\begin{align*}
		\int_A\Gamma(\psi,\psi)(dx)
		&=\int_{A\times(A\cup A^c)}
		(\psi(x)-\psi(y))^2\,J(dx,dy)\\
		&\le 2\int_{A\times A^c}(g(x)-g(y))^2\,J(dx,dy)
		+\int_{A^c\times A^c}(g(x)-g(y))^2\,J(dx,dy)\\
		&=\mathcal{E}(g,g).
	\end{align*}
	The proof in case (b) is similar, replacing $u$ with
	$v:=\psi\wedge(1-g)$ in the above argument. Note that
	$v=T(\psi,g)$ where $T:\mathbb{R}^2\to\mathbb{R}$ is the
	$1$-Lipschitz function defined by $T(a,b):=\min(a,1-b)$.
	Since $T(0,0)=0$, by the generalized contraction property
	\cite[Theorem A.2]{KS},
	\cite[Exercise I.3.5, Proof of Proposition I.3.3.1]{BH91},
	we have $v\in\mathcal{F}$. The same estimates as above apply.
\end{proof}
The following lemma is a variant of \cite[Lemma 6.1]{BCLS}.
\begin{lem} \label{l:truncate} 
	Let $(\mathcal{E},\mathcal{F})$ be a regular, pure jump Dirichlet form on $L^2(X,m)$ with jump measure $J(dx,dy)= J(x,y)\,m(dx)\,m(dy)$.  Let $f \in \mathcal{F}$ be a non-negative function. For any $k \in \mathbb{Z}$, define 
	\begin{equation} \label{e:def-fk}
		f_k:= (0 \vee (f-2^k)) \wedge 2^k.
	\end{equation}
	Then 
	\begin{equation} \label{e:truncate}	\sum_{k \in \mathbb{Z}} \Gamma(f_k,f_k)(x) \le \Gamma(f,f)(x).
	\end{equation}
\end{lem}
\begin{proof}
	Since $f(x), f(y) \ge 0$, we note that for any $x,y \in X$,
	\[
	\sum_{k \in \mathbb{Z}} \lvert f_k(x) - f_k(y) \rvert
	= \lvert f(x) - f(y) \rvert .
	\]
	The estimate between $\ell_1$ and $\ell_2$ norms is 
	\[
	\sum_{k \in \mathbb{Z}} (f_k(x) - f_k(y))^2 \le (f(x)-f(y))^2,
	\]
	and therefore 
	\begin{align*}
		\sum_{k \in \mathbb{Z}} \Gamma(f_k,f_k)(x) &= \int \sum_{k \in \mathbb{Z}} (f_k(x) - f_k(y))^2 J(x,y)\,m(dy) \\
		&\le \int  (f(x) - f(y))^2 J(x,y)\,m(dy) = \Gamma(f,f)(x).
	\end{align*}
\end{proof}

The following result is a version of cutoff Sobolev inequality for functions with somewhat restrictive boundary conditions. This restriction on the boundary condition will be later removed by the Whitney blending argument developed in Lemma \ref{l:wb}. 
\begin{lemma} \label{l:zerobdy-cs}
	Let $(\mathcal{E},\mathcal{F})$ be a regular, pure jump Dirichlet form on $L^2(X,m)$ with jump measure $J(dx,dy)=J(x,y)\,m(dx)\,m(dy)$.
	Let $\psi$ be an equilibrium potential for $K\subset\Omega$, where $\Omega$ is a non-polar, precompact open subset and $K$ is compact. Let $f\in\mathcal{F}\cap C(X)$ satisfy one of the following conditions:
	\begin{enumerate}[(a)]
		\item $f\equiv0$ on some neighborhood of $\Omega^c$, or
		\item $f\equiv0$ on some neighborhood of $K$.
	\end{enumerate}
	Then for any $U\subset X$,
	\begin{equation} \label{e:zero-bdy}
		\int_U f^2\,d\Gamma(\psi,\psi)
		\le 12\,\mathcal{E}(f,f).
	\end{equation}
\end{lemma}

\begin{proof}
	By replacing $f$ by $|f|\in\mathcal{F}$ and noting that
	$\Gamma(|f|,|f|)\le\Gamma(f,f)$, we assume that $f$ is non-negative without loss of generality \cite[Theorem 1.4.1(e)]{FOT}.
	
	For each $k\in\mathbb{Z}$, we define $f_k$ as given in
	\eqref{e:def-fk}, and set $g_k:=2^{-k}f_k$, so that
	\begin{equation} \label{e:zb0}
		g_k\equiv1
		\quad\mbox{on $A_k:=\{x\in U:f(x)\ge2^{k+1}\}$.}
	\end{equation}
	Note that
	\[
	f(x)^2\ge\sum_{k\in\mathbb{Z}}4^k\mathbf{1}_{A_k}(x)
	\ge f(x)^2/12,\quad\mbox{for all $x\in U$,}
	\]
	and hence
	\begin{equation} \label{e:zb1}
		\int_U f^2\,d\Gamma(\psi,\psi)
		\le 12\sum_{k\in\mathbb{Z}}4^k
		\Gamma(\psi,\psi)(A_k).
	\end{equation}
	By applying Lemma \ref{l:logcacc} to $g=g_k$ and $A=A_k$
	using \eqref{e:zb0}, we obtain
	\begin{equation} \label{e:zb2}
		\Gamma(\psi,\psi)(A_k)
		\le \mathcal{E}(g_k,g_k)
		=4^{-k}\mathcal{E}(f_k,f_k).
	\end{equation}
	Hence by \eqref{e:zb1}, \eqref{e:zb2}, and Lemma
	\ref{l:truncate}, we obtain
	\[
	\int_U f^2\,d\Gamma(\psi,\psi)
	\le 12\sum_{k\in\mathbb{Z}}\mathcal{E}(f_k,f_k)
	\le12\,\mathcal{E}(f,f).
	\]
\end{proof}
\subsection{Whitney cover}

We recall the notion of a $\epsilon$-Whitney cover from \cite[Definition 3.16]{GS}. For a proper open set $U \subsetneq X$, we denote by 
\[
\delta_U(x):= \dist(x, U^c) =\inf_{y \in U^c} d(x,y).
\]
\begin{definition} \label{d:whitney}
	Let $\epsilon \in (0,1/2)$ and $U \subsetneq X$ be open. 
	We say a collection of balls $\mathfrak{R}:= \set{B(x_i,r_i): x_i \in U, r_i >0, i \in I}$ is an $\epsilon$-Whitney cover if it satisfies the following properties:
	\begin{enumerate}[(i)]
		\item The collection of sets $\{B(x_i,r_i), i \in I \}$ are pairwise disjoint.
		\item The radii $r_i$ satisfy $$r_i= \frac{\epsilon}{1+\epsilon}\delta_U(x_i), \quad \mbox{for all $i \in I$.}$$
		\item $\bigcup_{i \in I} B(x_i, K_\eps r_i) = U$, where $K_\epsilon=2 (1+\epsilon) \in (2,3)$.
	\end{enumerate}
	Note that since  $K_\epsilon < 3$   by (iii), we have $$\bigcup_{i \in I} B_U(x_i, 3 r_i) = U.$$
\end{definition}
The existence of a Whitney cover follows by applying Zorn’s lemma to select
a maximal collection of balls satisfying properties (i) and (ii) in
Definition~\ref{d:whitney}.
By \cite[Proposition 3.2-(d)]{Mur24}, we have for any open set $U \subsetneq X$, and for any $\epsilon \in (0,1/2)$, there exists $C>0$ depending only on $\epsilon$ and the constants in \hyperlink{vd}{$\on{VD}$} such that for any $\epsilon$-Whitney cover $\{B(x_i,r_i): i \in I\}$, we have
\begin{equation} \label{e:w-overlap}
	\sum_{i \in I} \one_{B(x_i, r_i/\epsilon)}  \le C \one_U.
\end{equation}
We also record the following consequence of triangle inequality for future use. Let $\{B(x_i,r_i): i \in I\}$ be an $\epsilon$-Whitney cover of an open set $U \subsetneq X$, where $\epsilon \in (0,1/2)$. Then for any $\lambda >0$, for all $i \in I$ and for all $z \in B(x_i, \lambda r_i)$, we have 
\begin{equation} \label{e:tri-whi}
	\left(1 - \frac{\lambda \epsilon}{1+\epsilon}\right) \delta_U(x_i) <	\delta_U(z) <  \left(1 + \frac{\lambda \epsilon}{1+\epsilon}\right) \delta_U(x_i).
\end{equation}
The proof for the case $\lambda=3$ in \cite[Proposition 3.2-(b)]{Mur24} readily generalizes to yield \eqref{e:tri-whi}.

The following existence of partition of unity follows from the same argument as the proof of  Lemma \ref{l:partition}, where the desired bounded overlap property follows from \eqref{e:w-overlap}.
\begin{lem} \label{l:w-partition}
	Let $(X,d,m)$ be a metric measure space and let $(\mathcal{E},\mathcal{F})$ be a regular Dirichlet form that satisfies  \hyperlink{cap}{$\on{Cap}(\phi)_\le$} and let $m$ be a doubling measure. There exists $\epsilon_0 \in (0,1/6),C_1,c_1 > 0$ such that 
	for any open set $U \subsetneq X$ be open and for any $\epsilon$-Whitney cover  $\{B_i=B(x_i,r_i): i \in I\}$   of $U$ where $0<\epsilon < \epsilon_0$, there exists a partition on unity $\{\psi_{B_i}: i \in I\}$ of $U$ such that the following properties hold:
	\begin{enumerate}[(a)]
		\item (partition of unity) $\sum_{i \in I} \psi_{B_i} \equiv \one_U$. 
		\item (controlled energy) For each $i \in I$, we have $\psi_{B_i} \in \mathcal{F} \cap C(X)$ and
		\begin{equation}\label{e:w-pu}
			c_1 \one_{B(x_i,3r_i)} \le \psi_{B_i} \le \one_{B(x_i,6r_i)}, \quad	\mathcal{E}(\psi_{B_i},\psi_{B_i}) \le C_1 \frac{V(x_i,r_i)}{\phi(r_i)}.
		\end{equation}
	\end{enumerate}
\end{lem}

\subsection{Whitney blending} \label{ss:whitney}

The following \emph{Whitney blending} result is a non-local analogue of
\cite[Lemma 1.6]{Eri} and is the key ingredient behind our main result.
\begin{lem} \label{l:wb}
	Let $(X,d,m)$ be a metric measure space that satisfies
	\hyperlink{vd}{$\on{VD}$}. Let $(\mathcal{E},\mathcal{F})$
	be a regular pure jump Dirichlet form that satisfies
	\hyperlink{jphi}{$\on{J(\phi)}$} and
	\hyperlink{cap}{$\on{Cap(\phi)_{\le}}$}.
	For any $\eta\in(0,1)$, there exists $C>0$ depending only on
	the constants in $\on{VD}$, $\on{J(\phi)}$,
	$\on{Cap(\phi)_{\le}}$,   \eqref{e:scale},
	and $\eta$, such that, for any
	$x_0\in X$, $0<r<\diam(X,d)/4$, and
	$f,g\in\mathcal{F}\cap C(X)$, there exists
	$h\in\mathcal{F}\cap C(X)$ satisfying
	$h\equiv f$ on $B(x_0,r)$,
	$h\equiv g$ on $B(x_0,(1+\eta)r)^c$, and
	for every $A\in[2,\infty)$,
	\begin{align} \label{e:wb}
		\MoveEqLeft{
			\int_{B(x_0,Ar)\times B(x_0,Ar)}
			(h(x)-h(y))^2\,J(dx,dy)}
		\nonumber\\
		&\le C
		\int_{B(x_0,Ar)\times B(x_0,Ar)}
		\bigl((f(x)-f(y))^2+(g(x)-g(y))^2\bigr)\,J(dx,dy)
		\nonumber\\
		&\quad+\frac{C}{\phi(r)}
		\int_{B(x_0,Ar)}(f-g)^2\,dm.
	\end{align}
\end{lem}

\begin{proof}
	As explained in \cite[Proof of Lemma 1.6]{Eri}, it suffices
	to consider the case $g=0$. Indeed, apply this case to
	$\widetilde f=f-g$, obtaining $\widetilde h$, and set
	$h=\widetilde h+g$. The desired estimate then follows from
	\begin{align*}
		(\widetilde f(x)-\widetilde f(y))^2
		&\le 2\bigl((f(x)-f(y))^2+(g(x)-g(y))^2\bigr),\\
		(h(x)-h(y))^2
		&\le 2\bigl((\widetilde h(x)-\widetilde h(y))^2
		+(g(x)-g(y))^2\bigr).
	\end{align*}
	Thus we assume that $g=0$.

	Set
	\[
	K:=\overline{B(x_0,r)},\qquad
	F:=B(x_0,(1+\eta)r)^c,\qquad
	U:=X\setminus(K\cup F).
	\]
	Here $K$ denotes the closure of the open ball.
	We first consider $U\ne\emptyset$.
	
	Let $\{B_i=B(x_i,r_i):i\in I\}$ be an
	$\epsilon$-Whitney cover of $U$.
	
	Choose $\epsilon>0$ sufficiently small that
	$\epsilon<\epsilon_0$, where $\epsilon_0$ is the constant
	in Lemma \ref{l:w-partition}, and that all the following
	conditions hold:
	
	\begin{equation} \label{e:eps1}
		0<\epsilon<\frac{\eta}{36},
		\qquad
		\frac{12\epsilon(1+\eta/2)}{1-5\epsilon}
		<\frac{\eta}{6}.
	\end{equation}
	Define
	\begin{equation} \label{e:defA12}
		A_1:=\frac{7\epsilon}{1-5\epsilon}<\frac18,\qquad
		A_2:=\frac{A_1}{1-A_1}<\frac18,\qquad
		A_3:=\frac{6A_2}{(1-4A_2)(1-A_2)}<1,
	\end{equation}
	and
	\begin{equation} \label{e:defA4}
		A_4:=
		\frac{1+4\epsilon}{1+\epsilon}A_3
		+\frac{3\epsilon}{1+\epsilon},
		\qquad
		A_5:=A_4+\frac{7\epsilon(1+A_4)}{1-5\epsilon}
		<\frac{1}{1+\epsilon}.
	\end{equation}
	
	Such a choice is possible because $A_1,\ldots,A_5$
	tend to zero as $\epsilon\downarrow0$.
	Fix this choice for the rest of the proof.

	Define
	\[
	\widetilde I:=
	\left\{i\in I:
	B(x_i,6r_i)\cap B(x_0,(1+\eta/2)r)\ne\emptyset
	\right\}.
	\]
	Let $\{\psi_{B_i}:i\in I\}$ be the partition of unity
	from Lemma \ref{l:w-partition}.

	For $i\in\widetilde I$, since $x_0\in U^c$, we have
	\[
	\delta_U(x_i)\le d(x_i,x_0)
	<(1+\eta/2)r+6r_i.
	\]
	Together with $r_i=\epsilon\delta_U(x_i)/(1+\epsilon)$,
	this implies
	\[
	r_i<
	\frac{\epsilon(1+\eta/2)}{1-5\epsilon}\,r.
	\]
	Consequently, by \eqref{e:eps1},
	\[
	B(x_i,6r_i)
	\subset B\bigl(x_0,(1+\eta/2)r+12r_i\bigr)
	\subset B(x_0,(1+2\eta/3)r).
	\]
	
	It follows that
	\begin{equation} \label{e:wtI}
		\one_{B(x_0,(1+\eta/2)r)\setminus K}
		\le \sum_{i\in\widetilde I}\psi_{B_i}
		\le \one_{B(x_0,(1+2\eta/3)r)\setminus K}.
	\end{equation}
	
	Define
	\begin{equation} \label{e:def-ext}
		h(x):=
		\begin{cases}
			f(x),&x\in K,\\
			\displaystyle\sum_{i\in\widetilde I}
			f_{B_i}\psi_{B_i}(x),&x\notin K.
		\end{cases}
	\end{equation}
	Then $h=f$ on $K$ and
	$h=0$ on $B(x_0,(1+2\eta/3)r)^c$.
	
	We first check that $h\in L^2(X,m)$.
	
	Since $\sum_{i\in\widetilde I}\psi_{B_i}\le1$,
	the Cauchy--Schwarz inequality gives
	\[
	|h(x)|^2
	\le\sum_{i\in\widetilde I}|f_{B_i}|^2\psi_{B_i}(x),
	\qquad x\in U.
	\]
	Thus, by \eqref{e:w-pu}, volume doubling, Jensen's
	inequality, and the disjointness of the Whitney balls,
	\[
	\int_U h^2\,dm
	\lesssim
	\sum_{i\in\widetilde I}|f_{B_i}|^2m(B_i)
	\le \sum_{i\in\widetilde I}\int_{B_i}f^2\,dm
	\le\int_U f^2\,dm.
	\]
	
	In particular,
	\begin{equation} \label{e:wb0}
		\int_X h^2\,dm
		\lesssim\int_{K\cup U}f^2\,dm
		=\int_{B(x_0,(1+\eta)r)}f^2\,dm<\infty.
	\end{equation}
	
	Next, we prove
	\begin{equation} \label{e:wb1}
		\int_{X\times X}(h(x)-h(y))^2\,J(dx,dy)<\infty.
	\end{equation}
	By symmetry, it suffices to consider the following
	five cases.
	
	\noindent\textbf{Case 1: $x,y\in K$.}
	Since $h=f$ on $K$,
	\begin{equation} \label{e:wb2}
		\int_{K\times K}(h(x)-h(y))^2\,J(dx,dy)
		=
		\int_{K\times K}(f(x)-f(y))^2\,J(dx,dy).
	\end{equation}
	
	\noindent\textbf{Case 2: $x,y\in F$.}
	Since $h=0$ on $F$,
	\begin{equation} \label{e:wb3}
		\int_{F\times F}(h(x)-h(y))^2\,J(dx,dy)=0.
	\end{equation}
	
	\noindent\textbf{Case 3: $x,y\in U$.}
	Set
	\[
	S(x):=\sum_{i\in\widetilde I}\psi_{B_i}(x),
	\qquad
	\widetilde I(x):=
	\{i\in\widetilde I:\psi_{B_i}(x)\ne0\}.
	\]
	By \eqref{e:wtI}, $0\le S\le1$ on $U$, and
	$S=1$ on $B(x_0,(1+\eta/2)r)\setminus K$.
	Moreover, since $\psi_{B_i}$ vanishes outside $6B_i$,
	\eqref{e:w-overlap} gives
	$\sup_{x\in U}\#\widetilde I(x)\lesssim1$.
	
	When $S(x)=S(y)=1$, we have
	\begin{align} \label{e:wb4}
		(h(x)-h(y))^2
		&=
		\left(
		\sum_{\substack{i\in\widetilde I(x)\\
				j\in\widetilde I(y)}}
		(f_{B_i}-f_{B_j})\psi_{B_i}(x)\psi_{B_j}(y)
		\right)^2\lesssim
		\sum_{\substack{i\in\widetilde I(x)\\
				j\in\widetilde I(y)}}
		(f_{B_i}-f_{B_j})^2.
	\end{align}
	For $i\in\widetilde I(x)$, the triangle inequality gives
	\begin{equation} \label{e:wb5}
		B_i\subset B(x,7r_i)
		\subset B(x,A_1\delta_U(x)).
	\end{equation}
	Hence Jensen's inequality, volume doubling, and
	\eqref{e:wb4}--\eqref{e:wb5} imply
	\begin{equation} \label{e:wb6}
		(h(x)-h(y))^2
		\lesssim
		\int_{B(x,A_1\delta_U(x))}
		\int_{B(y,A_1\delta_U(y))}
		\frac{(f(z)-f(w))^2}
		{V(z,A_2\delta_U(z))V(w,A_2\delta_U(w))}
		\,m(dw)\,m(dz)
	\end{equation}
	whenever $S(x)=S(y)=1$.
	
	For $z\in B(x,A_1\delta_U(x))$,
	\begin{equation} \label{e:te1}
		(1-A_1)\delta_U(x)
		\le\delta_U(z)\le(1+A_1)\delta_U(x),
		\qquad
		\one_{B(x,A_1\delta_U(x))}(z)
		\le\one_{B(z,A_2\delta_U(z))}(x).
	\end{equation}
	If
	\[
	d(z,w)<4A_2(\delta_U(z)\vee\delta_U(w)),\qquad
	x\in B(z,A_2\delta_U(z)),\qquad
	y\in B(w,A_2\delta_U(w)),
	\]
	then
	\begin{align*}
		d(x,y)
		&<6A_2(\delta_U(z)\vee\delta_U(w))\\
		&\le\frac{6A_2}{1-4A_2}
		(\delta_U(z)\wedge\delta_U(w))\\
		&\le\frac{6A_2}{(1-4A_2)(1-A_2)}
		(\delta_U(x)\wedge\delta_U(y)).
	\end{align*}
	Therefore, when
	$d(x,y)\ge A_3(\delta_U(x)\wedge\delta_U(y))$,
	$z\in B(x,A_1\delta_U(x))$, and
	$w\in B(y,A_1\delta_U(y))$, we have
	\begin{equation} \label{e:te2}
		A_2(\delta_U(z)\vee\delta_U(w))\le d(z,w)/4.
	\end{equation}
	
	Put $\theta:=\eta/12$ and
	\[
	Q:=
	\left\{(x,y)\in U\times U:
	\begin{array}{l}
		d(x,y)<\theta r,\quad S(x)=S(y)=1,\\
		d(x,y)\ge A_3(\delta_U(x)\wedge\delta_U(y))
	\end{array}
	\right\}.
	\]
	Using \eqref{e:wb6}, \eqref{e:te1}, \eqref{e:te2},
	Tonelli’s theorem, and Lemma \ref{l:ujs}, we obtain
	\begin{equation} \label{e:wb7a}
		\int_Q(h(x)-h(y))^2\,J(dx,dy)
		\lesssim
		\int_{U\times U}(f(z)-f(w))^2\,J(dz,dw).
	\end{equation}
	
	For pairs with $d(x,y)\ge\theta r$, the jump-tail
	estimate and the estimate for $\int_Uh^2\,dm$ above
	give
	\begin{align} \label{e:wb7b}
		&\int_{U\times U}
		\one_{\{d(x,y)\ge\theta r\}}
		(h(x)-h(y))^2\,J(dx,dy)
		\nonumber\\
		&\qquad\lesssim\frac1{\phi(r)}\int_U h^2\,dm
		\lesssim\frac1{\phi(r)}\int_U f^2\,dm.
	\end{align}
	
	Consider next pairs with $d(x,y)<\theta r$ and
	$S(x)<1$ or $S(y)<1$.
	If, for example, $S(x)<1$, then
	$d(x_0,x)\ge(1+\eta/2)r$, so both $z=x,y$ satisfy
	$d(x_0,z)>(1+5\eta/12)r$.
	If $j\in\widetilde I$ and $\psi_{B_j}(z)\ne0$,
	then \eqref{e:wtI} gives
	$d(x_0,z)<(1+2\eta/3)r$.
	Consequently,
	\[
	\delta_U(z)\ge\eta r/3,
	\qquad
	r_j\ge a_0r,
	\qquad
	a_0:=\frac{\epsilon\eta}{3(1+7\epsilon)}>0.
	\]
	The same argument applies if $S(y)<1$.
	Writing
	\[
	h(x)-h(y)=
	\sum_{j\in\widetilde I}
	f_{B_j}\bigl(\psi_{B_j}(x)-\psi_{B_j}(y)\bigr)
	\]
	and using the uniformly bounded number of nonzero terms,
	we obtain
	\begin{align} \label{e:wb7c}
		&\int_{\substack{(x,y)\in U\times U,\ d(x,y)<\theta r\\
				S(x)<1\ \text{or}\ S(y)<1}}
		(h(x)-h(y))^2\,J(dx,dy)
		\nonumber\\
		&\qquad\lesssim
		\sum_{\substack{j\in\widetilde I\\r_j\ge a_0r}}
		f_{B_j}^{\,2}\mathcal E(\psi_{B_j},\psi_{B_j})
		\nonumber\\
		&\qquad\lesssim
		\frac1{\phi(r)}
		\sum_{j\in\widetilde I}f_{B_j}^{\,2}m(B_j)
		\lesssim\frac1{\phi(r)}\int_U f^2\,dm.
	\end{align}
	Combining \eqref{e:wb7a}--\eqref{e:wb7c} yields
	\begin{align} \label{e:wb7}
		&\int_{U\times U}
		\one_{\{d(x,y)\ge A_3(\delta_U(x)\wedge\delta_U(y))\}}
		(h(x)-h(y))^2\,J(dx,dy)
		\nonumber\\
		&\qquad\lesssim
		\int_{U\times U}(f(x)-f(y))^2\,J(dx,dy)
		+\frac1{\phi(r)}\int_U f^2\,dm.
	\end{align}
	
	We now estimate the remaining pairs.
	
	For each $i\in I$, write
	\[
	\delta_i:=\delta_U(x_i),\qquad
	C_i:=B(x_i,A_4\delta_i),\qquad
	D_i:=B(x_i,A_5\delta_i).
	\]
	
	If $x\in B(x_i,3r_i)$, then
	\[
	\delta_U(x)\le
	\frac{1+4\epsilon}{1+\epsilon}\delta_i.
	\]
	Thus $d(x,y)<A_3\delta_U(x)$ implies $x,y\in C_i$.
	Since $\bigcup_i B(x_i,3r_i)=U$,
	\begin{equation} \label{e:wb8}
		\one_{\{d(x,y)<A_3(\delta_U(x)\wedge\delta_U(y))\}}
		\le\sum_{i\in I}\one_{C_i}(x)\one_{C_i}(y).
	\end{equation}
	Consequently,
	\begin{align} \label{e:wb9}
		&\int_{U\times U}
		\one_{\{d(x,y)<A_3(\delta_U(x)\wedge\delta_U(y))\}}
		(h(x)-h(y))^2\,J(dx,dy)
		\nonumber\\
		&\qquad\le
		\sum_{i\in I}\int_{C_i\times C_i}
		(h(x)-h(y))^2\,J(dx,dy).
	\end{align}

	Define the active index set
	\[
	\mathcal J_i:=
	\{j\in\widetilde I:B(x_j,6r_j)\cap C_i\ne\emptyset\}.
	\]
	For $j\in\mathcal J_i$, choose
	$z\in B(x_j,6r_j)\cap C_i$.
	The triangle inequality gives
	\[
	\frac{1+\epsilon}{1+7\epsilon}(1-A_4)\delta_i
	<\delta_U(x_j)
	<
	\frac{1+\epsilon}{1-5\epsilon}(1+A_4)\delta_i.
	\]
	Moreover, for every $w\in B_j$,
	\[
	d(w,x_i)<A_4\delta_i+7r_j<A_5\delta_i,
	\]
	so $B_j\subset D_i$.
	In particular, with constants depending on the fixed
	$\epsilon$,
	\[
	r_j\asymp\delta_i,\qquad
	m(B_j)\asymp m(D_i),\qquad
	\phi(r_j)\asymp\phi(\delta_i).
	\]
	The disjointness of the Whitney balls and volume doubling
	therefore imply
	\[
	\sup_{i\in I}\#\mathcal J_i\lesssim1.
	\]
	Also, \eqref{e:defA4} and \eqref{e:w-overlap} give
	\[
	\sum_{i\in I}\one_{D_i}\lesssim\one_U.
	\]
	Since $j\in\mathcal J_i$ implies $x_j\in D_i$, we also have
	\[
	\sup_{j\in\widetilde I}
	\#\{i\in I:j\in\mathcal J_i\}\lesssim1.
	\]

	Set
	\[
	I_1:=\{i\in I:C_i\subset B(x_0,(1+\eta/2)r)\}.
	\]
	For $i\in I_1$, we have $S=1$ on $C_i$.
	Therefore, for $x,y\in C_i$,
	\begin{align} \label{e:wb10}
		(h(x)-h(y))^2
		&=
		\left(
		\sum_{j\in\mathcal J_i}
		(f_{B_j}-f_{D_i})
		(\psi_{B_j}(x)-\psi_{B_j}(y))
		\right)^2
		\nonumber\\
		&\lesssim
		\sum_{j\in\mathcal J_i}
		(f_{B_j}-f_{D_i})^2
		(\psi_{B_j}(x)-\psi_{B_j}(y))^2.
	\end{align}

	By Jensen's inequality, the lower bound on $J$,
	volume doubling, and the scale condition,
	\begin{align*}
		|f_{B_j}-f_{D_i}|^2
		&\le
		\frac1{m(B_j)m(D_i)}
		\int_{B_j}\int_{D_i}(f(z)-f(w))^2\,m(dw)\,m(dz)\\
		&\lesssim
		\frac{\phi(\delta_i)}{m(B_j)}
		\int_{D_i\times D_i}(f(z)-f(w))^2\,J(dz,dw).
	\end{align*}
	Multiplying by
	$\mathcal E(\psi_{B_j},\psi_{B_j})
	\lesssim m(B_j)/\phi(r_j)$,
	using $\phi(r_j)\asymp\phi(\delta_i)$, and summing over
	$\mathcal J_i$, we obtain from \eqref{e:wb10}
	
	\begin{align} \label{e:wb11}
		\sum_{i\in I_1}\int_{C_i\times C_i}
		(h(x)-h(y))^2\,J(dx,dy)
		&\lesssim
		\sum_{i\in I_1}\int_{D_i\times D_i}
		(f(x)-f(y))^2\,J(dx,dy)
		\nonumber\\
		&\lesssim
		\int_{U\times U}(f(x)-f(y))^2\,J(dx,dy).
	\end{align}

	Next let $i\notin I_1$ and $\mathcal J_i\ne\emptyset$.
	Choose
	$q\in C_i\setminus B(x_0,(1+\eta/2)r)$.
	For $j\in\mathcal J_i$, choose
	$z\in C_i\cap B(x_j,6r_j)$.
	The support inclusion proved above gives
	$d(x_0,z)<(1+2\eta/3)r$.
	It follows that
	\[
	\dist(x_i,K)\ge\eta r/2-A_4\delta_i,
	\qquad
	\dist(x_i,F)\ge\eta r/3-A_4\delta_i.
	\]
	Taking the minimum yields
	\[
	\delta_i\ge\frac{\eta r}{3(1+A_4)}.
	\]
	Consequently, for every $j\in\mathcal J_i$,
	\[
	r_j>ar,\qquad
	a:=\frac{\epsilon(1-A_4)\eta}
	{3(1+7\epsilon)(1+A_4)}>0.
	\]
	
	For $x,y\in C_i$, we use
	\begin{align} \label{e:wb12}
		(h(x)-h(y))^2
		&=
		\left(
		\sum_{j\in\mathcal J_i}
		f_{B_j}(\psi_{B_j}(x)-\psi_{B_j}(y))
		\right)^2
		\nonumber\\
		&\lesssim
		\sum_{j\in\mathcal J_i}
		f_{B_j}^{\,2}
		(\psi_{B_j}(x)-\psi_{B_j}(y))^2.
	\end{align}
	
	The bounded multiplicity of the sets $\mathcal J_i$,
	established above, now gives
	
	\begin{align} \label{e:wb13}
		\sum_{i\notin I_1}\int_{C_i\times C_i}
		(h(x)-h(y))^2\,J(dx,dy)
		&\lesssim
		\sum_{\substack{j\in\widetilde I\\r_j>ar}}
		f_{B_j}^{\,2}\mathcal E(\psi_{B_j},\psi_{B_j})
		\nonumber\\
		&\lesssim\frac1{\phi(r)}
		\sum_{j\in\widetilde I}f_{B_j}^{\,2}m(B_j)
		\lesssim\frac1{\phi(r)}\int_U f^2\,dm.
	\end{align}
	Combining \eqref{e:wb7}, \eqref{e:wb9},
	\eqref{e:wb11}, and \eqref{e:wb13}, we obtain
	\begin{align} \label{e:wb14}
		\int_{U\times U}(h(x)-h(y))^2\,J(dx,dy)
		&\lesssim
		\int_{U\times U}(f(x)-f(y))^2\,J(dx,dy)
		\nonumber\\
		&\quad+\frac{1}{\phi(r)}\int_U f^2\,dm.
	\end{align}
	
	\noindent\textbf{Case 4: $x\in K$, $y\in U$.}
	First suppose that
	$y\in B(x_0,(1+\eta/2)r)\setminus K$.
	Then $S(y)=1$, and hence
	\[
	(h(x)-h(y))^2
	=
	\left(
	\sum_{i\in\widetilde I(y)}
	(f(x)-f_{B_i})\psi_{B_i}(y)
	\right)^2
	\lesssim
	\sum_{i\in\widetilde I(y)}(f(x)-f_{B_i})^2.
	\]
	Using \eqref{e:wb5}, Jensen's inequality, and volume
	doubling, we get
	\begin{align} \label{e:wb15}
		(h(x)-h(y))^2
		&\lesssim
		\int_{B(y,A_1\delta_U(y))}
		\frac{(f(x)-f(z))^2}{V(y,A_1\delta_U(y))}\,m(dz)
		\nonumber\\
		&\lesssim
		\int_{B(y,A_1\delta_U(y))}
		\frac{(f(x)-f(z))^2}{V(z,2A_2\delta_U(z))}\,m(dz).
	\end{align}
	Moreover,
	\begin{equation} \label{e:wb15a}
		\one_{B(y,A_1\delta_U(y))}(z)
		\le\one_{B(z,2A_2\delta_U(z))}(y).
	\end{equation}
	Since $\delta_U(z)\le d(x,z)$ for $x\in K$ and
	$A_2<1/8$, Tonelli’s theorem and Lemma \ref{l:ujs}
	therefore give
	\begin{align} \label{e:wb16}
		&\int_K\int_{B(x_0,(1+\eta/2)r)\setminus K}
		(h(x)-h(y))^2\,J(x,y)\,m(dy)\,m(dx)
		\nonumber\\
		&\qquad\lesssim
		\int_K\int_U
		\frac{(f(x)-f(z))^2}{V(z,2A_2\delta_U(z))}
		\int_{B(z,2A_2\delta_U(z))}J(x,y)\,m(dy)
		\,m(dz)\,m(dx)
		\nonumber\\
		&\qquad\lesssim
		\int_K\int_U(f(x)-f(z))^2J(x,z)\,m(dz)\,m(dx).
	\end{align}

	For the remaining pairs, write
	$U_{\rm out}:=U\setminus B(x_0,(1+\eta/2)r)$.
	If $x\in K$ and $y\in U_{\rm out}$, then
	$d(x,y)\ge\eta r/2$. Thus
	
	\begin{equation} \label{e:wb17}
		(h(x)-h(y))^2
		\le2f(x)^2+2h(y)^2.
	\end{equation}
	
	By symmetry, the jump-tail estimate, and
	$\int_Uh^2\,dm\lesssim\int_Uf^2\,dm$, this implies
	
	\begin{align} \label{e:wb18}
		\int_K\int_{U_{\rm out}}
		(h(x)-h(y))^2J(x,y)\,m(dy)\,m(dx)
		&
		\lesssim\frac1{\phi(r)}
		\left(\int_Kf^2\,dm+\int_Uh^2\,dm\right)
		\nonumber\\
		&
		\lesssim\frac1{\phi(r)}\int_{K\cup U}f^2\,dm.	
	\end{align}
	Combining \eqref{e:wb16} and \eqref{e:wb18},
	\begin{align} \label{e:wb19}
		\int_K\int_U(h(x)-h(y))^2J(x,y)\,m(dy)\,m(dx)
		&\lesssim
		\int_K\int_U(f(x)-f(y))^2J(x,y)\,m(dy)\,m(dx)
		\nonumber\\
		&\quad+\frac1{\phi(r)}\int_{K\cup U}f^2\,dm.
	\end{align}
	
	\noindent\textbf{Case 5: $x\in K\cup U$, $y\in F$.}
	Here $h(y)=0$. If $h(x)\ne0$, the support bound gives
	$d(x,y)\ge\eta r/3$. Consequently,
	\begin{align} \label{e:wb20}
		\int_{K\cup U}\int_F
		(h(x)-h(y))^2J(x,y)\,m(dy)\,m(dx)
		&\le
		\int_{K\cup U}h(x)^2
		\int_{B(x,\eta r/3)^c}J(x,y)\,m(dy)\,m(dx)
		\nonumber\\
		&\lesssim\frac1{\phi(r)}\int_{K\cup U}h^2\,dm
		\nonumber\\
		&\lesssim\frac1{\phi(r)}\int_{K\cup U}f^2\,dm.
	\end{align}
	
	Combining the five cases gives
	\begin{align*}
		\int_{X\times X}(h(x)-h(y))^2\,J(dx,dy)
		&\lesssim
		\int_{B(x_0,(1+\eta)r)\times B(x_0,(1+\eta)r)}
		(f(x)-f(y))^2\,J(dx,dy)\\
		&\quad+\frac1{\phi(r)}
		\int_{B(x_0,(1+\eta)r)}f^2\,dm.
	\end{align*}
	By Proposition \ref{p:domain}, $h\in\mathcal F$.
	Since $1+\eta<2\le A$, this also proves \eqref{e:wb}
	for $g=0$.

	If $U=\emptyset$, define $h=f$ on $K$ and $h=0$ on $F$.
	Then $X=K\cup F$ and $\dist(K,F)\ge\eta r$.
	The estimates in Cases 1, 2, and 5 give the same
	global energy bound, so Proposition \ref{p:domain}
	again yields $h\in\mathcal F$.

	It remains to establish continuity.
	
	When $U\ne\emptyset$, the family
	$\{B(x_i,6r_i):i\in I\}$ is locally finite in $U$.
	Indeed, near any point of $U$, the Whitney radii of
	intersecting balls are bounded above and below by
	positive multiples of the distance of that point
	to $U^c$. Disjointness of the $B_i$ and volume doubling
	then give a finite number of such balls.
	Thus $h$ is continuous on $U$.
	The support bound \eqref{e:wtI} shows that $h$ vanishes
	on a neighborhood of $F$, so continuity across the
	outer boundary follows as well.

	Fix $\xi\in\partial K$.
	For sufficiently small $s>0$, every $y\in U\cap B(\xi,s)$
	belongs to $B(x_0,(1+\eta/2)r)$, and hence $S(y)=1$.
	For each $j\in\widetilde I(y)$, \eqref{e:wb5} gives
	\[
	B_j\subset B(y,A_1\delta_U(y))
	\subset B(\xi,(1+A_1)s),
	\]
	because $\xi\in U^c$ and therefore
	$\delta_U(y)\le d(y,\xi)<s$.
	It follows that
	\[
	|h(y)-f(\xi)|
	\le
	\sum_{j\in\widetilde I(y)}
	\psi_{B_j}(y)\fint_{B_j}|f(z)-f(\xi)|\,m(dz)
	\le
	\sup_{z\in B(\xi,(1+A_1)s)}|f(z)-f(\xi)|.
	\]
	The same upper bound holds for $y\in K\cap B(\xi,s)$,
	since $h=f$ on $K$.
	Letting $s\downarrow0$ proves continuity at $\xi$.
	If $U=\emptyset$, continuity follows directly from
	$\dist(K,F)\ge\eta r$.
	
	Thus $h\in\mathcal F\cap C(X)$, completing the proof.
\end{proof}
\subsection{Self-improvement of a simplified cutoff Sobolev inequality} \label{ss:self-improve}

We prove the following   version of cutoff Sobolev inequality using the Whitney blending result. This result can be viewed as a proof of the non-local analogue of simplified cutoff Sobolev inequality introduced in \cite[Definition 6.1]{Mur24} using the Whitney blending method of \cite{Eri}.

\begin{prop} \label{p:css}
	Let $(X,d,m)$ satisfy \hyperlink{vd}{$\operatorname{VD}$}, and let
	$(\mathcal{E},\mathcal{F})$ be a regular, pure jump Dirichlet form on
	$L^2(X,m)$ that satisfies \hyperlink{jphi}{$\operatorname{J(\phi)}$}
	and \hyperlink{cap}{$\operatorname{Cap(\phi)_{\le}}$}, where
	$\phi:[0,\infty)\to[0,\infty)$ is a scale function. Then there exists
	$C>0$ such that for any
	$0<r<\diam(X,d)/16$ and $f\in C(X)\cap\mathcal{F}$, any
	equilibrium potential $\psi$ for
	$\overline{B(x,r)}\subset B(x,2r)$ for the truncated Dirichlet form
	$(\mathcal{E}^{(r)},\mathcal{F})$ satisfies
	\begin{align} \label{e:css}
		\int_X f^2\,d\Gamma^{(r)}(\psi,\psi)
		&=\int_{B(x,3r)} f^2\,d\Gamma^{(r)}(\psi,\psi)
		\nonumber\\
		&\le C\int_{B(x,12r)}\int_{B(x,12r)}
		(f(y)-f(z))^2J(y,z)\,m(dy)\,m(dz)
		\nonumber\\
		&\quad+\frac{C}{\phi(r)}
		\int_{B(x,12r)}f^2\,dm.
	\end{align}
\end{prop}
\begin{proof}
	Since $\psi=0$ almost everywhere outside $B(x,2r)$
	and $J^{(r)}(y,z)=0$ when $d(y,z)>r$, the measure
	$\Gamma^{(r)}(\psi,\psi)$ vanishes outside $B(x,3r)$.
	This proves the equality in \eqref{e:css}.
	
	Apply Lemma \ref{l:wb} with $g=0$ to obtain
	$f_1\in\mathcal F\cap C(X)$ such that
	\begin{equation} \label{e:css1}
		f_1=f\ \text{on }\overline{B(x,3r/2)},\qquad
		f_1=0\ \text{on }B(x,7r/4)^c.
	\end{equation}
	Set $b:=f-f_1$. Apply Lemma \ref{l:wb} to $b$ and $0$
	to obtain $f_2\in\mathcal F\cap C(X)$ such that
	\[
	f_2=b\ \text{on }\overline{B(x,4r)},\qquad
	f_2=0\ \text{on }B(x,5r)^c.
	\]
	Both applications are permitted because
	$4r<\diam(X,d)/4$.
	
	Using $A=8$ in the first application and $A=3$ in
	the second, Lemma \ref{l:wb} and the scale condition
	\eqref{e:scale} give, for $i=1,2$,
	\begin{equation} \label{e:css2}
		\begin{aligned}
			&\int_{B(x,12r)\times B(x,12r)}
			(f_i(y)-f_i(z))^2\,J(dy,dz)
			\\
			&\quad\lesssim
			\int_{B(x,12r)\times B(x,12r)}
			(f(y)-f(z))^2\,J(dy,dz)
			+\frac1{\phi(r)}\int_{B(x,12r)}f^2\,dm.
		\end{aligned}
	\end{equation}
	For $i=2$, we also use $b=f-f_1$ and the bound
	\[
	\int_X f_1^2\,dm
	\lesssim \int_{B(x,7r/4)}f^2\,dm
	\]
	from \eqref{e:wb0} to control the energy and
	$L^2$ norm of $b$.
	
	On $B(x,3r)$, we have $f_2=f-f_1$, and hence
	\begin{equation} \label{e:css3}
		f^2\le2(f_1^2+f_2^2)
		\quad\text{on }B(x,3r).
	\end{equation}
	Moreover, $f_1$ vanishes on a neighborhood of
	$B(x,2r)^c$, while $f_2$ vanishes on a neighborhood
	of $\overline{B(x,r)}$, since $f_2=b=0$ on
	$B(x,3r/2)$. Thus both functions satisfy a boundary
	condition in Lemma \ref{l:zerobdy-cs}.
	
	Both functions vanish outside $B(x,5r)$, so every
	pair contributing to their $r$-truncated energies
	lies in $B(x,6r)\times B(x,6r)$. Applying
	Lemma \ref{l:zerobdy-cs} to the truncated form yields
	\begin{align} \label{e:css4}
		\int_X f^2\,d\Gamma^{(r)}(\psi,\psi)
		&\le
		2\sum_{i=1}^2
		\int_X f_i^2\,d\Gamma^{(r)}(\psi,\psi)
		\nonumber\\
		&\le
		24\sum_{i=1}^2\mathcal E^{(r)}(f_i,f_i)
		\nonumber\\
		&\le
		24\sum_{i=1}^2
		\int_{B(x,12r)\times B(x,12r)}
		(f_i(y)-f_i(z))^2\,J(dy,dz).
	\end{align}
	Combining \eqref{e:css2} and \eqref{e:css4}
	proves \eqref{e:css}.
\end{proof}

In view of Proposition~\ref{p:css}, the argument below, completing the proof of Theorem~\ref{t:main},
may be regarded as a non-local analogue of \cite[Lemma~6.2]{Mur24}.
As shown below, the truncated Dirichlet form plays a crucial role in adapting the argument of \cite{Mur24}
to the non-local setting.
\begin{proof}[Proof of Theorem \ref{t:main}]
	Let $x \in X$, $0<r<R$ be arbitrary. If $B(x,R+r)=X$, we
	choose $\psi\equiv 1$ as the cutoff function. In this case
	$\Gamma(\psi,\psi)\equiv 0$.
	
	It remains to consider the case $B(x,R+r)\neq X$. In this case
	$2r<R+r\le\diam(X,d)$ implies $r<\diam(X,d)/2$. Define
	\begin{equation} \label{e:defL}
		\rho:=r/100.
	\end{equation}
	Let $N=\{z_i:i\in I\}$ denote a maximal $\rho$-separated subset
	of $X$. For each $i\in I$, let
	$B_i=\ol{B}(z_i,\rho)$, $B_i^*=B(z_i,2\rho)$, and
	$B_i^\#=B(z_i,12\rho)$.
	By Proposition \ref{p:css}, there exists $C_1>0$ such that for
	any $i\in I$ and for all $f\in\mathcal{F}\cap C(X)$, the
	equilibrium potential $\psi_i$ for $B_i\subset B_i^*$
	corresponding to the truncated Dirichlet form
	$(\mathcal{E}^{(\rho)},\mathcal{F})$ satisfies
	\begin{equation} \label{e:si2}
		\int_X f^2\,d\Gamma^{(\rho)}(\psi_i,\psi_i)
		\le C_1\int_{B_i^\#\times B_i^\#}
		(f(y)-f(z))^2\,J(dy,dz)
		+\frac{C_1}{\phi(\rho)}
		\int_{B_i^\#}f^2\,dm.
	\end{equation}
	
	Define $\wt I\subset I$ by
	\[
	\wt I:=\{i\in I:
	B(z_i,\rho)\cap B(x,R+r/2)\neq\emptyset\},
	\]
	so that
	\begin{equation} \label{e:csj1}
		\begin{gathered}
			\bigcup_{i\in\wt I}B_i\supset B(x,R+r/2),\qquad
			\bigcup_{i\in\wt I}B_i^*
			\subset B(x,R+r/2+3\rho),\\
			\bigcup_{i\in\wt I}B_i^\#
			\subset B(x,R+r/2+13\rho).
		\end{gathered}
	\end{equation}
	The set $\widetilde I$ is finite, since its centers lie
	in a bounded ball and are $\rho$-separated. Therefore the function
	\[
	\psi:=\max_{i\in\wt I}\psi_i
	\]
	satisfies (by \cite[Theorem 1.4.2]{FOT}(i),(ii) and the
	triangle inequality)
	\[
	\psi\in\mathcal F,\quad 0\le\psi\le1,\quad
	\psi\equiv1\ \mbox{q.e.~on }B(x,R+r/2),\quad
	\psi\equiv0\ \mbox{q.e.~on }
	B(x,R+r/2+3\rho)^c,
	\]
	and hence
	\begin{equation} \label{e:csj2}
		\Gamma^{(\rho)}(\psi,\psi)(X)
		=\Gamma^{(\rho)}(\psi,\psi)
		\bigl(B(x,R+r/2+4\rho)
		\setminus B(x,R+r/2-\rho)\bigr).
	\end{equation}
	
	Let
	\[
	\wh I:=\{i\in\wt I:
	B_i^*\cap
	\bigl(B(x,R+r/2+5\rho)
	\setminus B(x,R+r/2-2\rho)\bigr)
	\neq\emptyset\}.
	\]
	Define $\wh\psi:=\max_{i\in\wh I}\psi_i$, with
	$\wh\psi:=0$ if $\wh I$ is empty. Then
	\[
	\psi=\wh\psi
	\quad\mbox{on }
	B(x,R+r/2+5\rho)\setminus B(x,R+r/2-2\rho).
	\]
	Consequently, the energy measures
	$\Gamma^{(\rho)}(\psi,\psi)$ and
	$\Gamma^{(\rho)}(\wh\psi,\wh\psi)$ coincide on
	$B(x,R+r/2+4\rho)\setminus B(x,R+r/2-\rho)$; that is,
	\begin{equation} \label{e:csj3}
		\Gamma^{(\rho)}(\psi,\psi)(A)
		=\Gamma^{(\rho)}(\wh\psi,\wh\psi)(A)
	\end{equation}
	for all
	$A\subset B(x,R+r/2+4\rho)\setminus B(x,R+r/2-\rho)$.
	Furthermore, by the triangle inequality and the choice of
	$\rho$ in \eqref{e:defL},
	\begin{equation} \label{e:csj4}
		\bigcup_{i\in\wh I}B_i^\#
		\subset
		B(x,R+r/2+19\rho)
		\setminus B(x,R+r/2-16\rho)
		\subset
		B(x,R+3r/4)\setminus B(x,R+r/4).
	\end{equation}
	
	For any $n\in\mathbb N$ and
	$a_1,\ldots,a_n,b_1,\ldots,b_n\in\mathbb R$, we have
	\[
	\left(\max_{1\le i\le n}a_i
	-\max_{1\le i\le n}b_i\right)^2
	\le \max_{1\le i\le n}(a_i-b_i)^2
	\le \sum_{i=1}^n(a_i-b_i)^2.
	\]
	Hence, integrating this pointwise bound gives
	\begin{equation} \label{e:nlsi1}
		\Gamma^{(\rho)}(\wh\psi,\wh\psi)
		\le\sum_{i\in\wh I}
		\Gamma^{(\rho)}(\psi_i,\psi_i).
	\end{equation}
	
	By \eqref{e:trunc-est-em}, \eqref{e:nlsi1}, and
	\eqref{e:scale}, there exists $C_2>0$ such that for any
	$A\subset X$,
	\begin{align} \label{e:nlsi3}
		\Gamma(\psi,\psi)(A)
		&\le \Gamma^{(\rho)}(\psi,\psi)(A)
		+\frac{C_2}{\phi(r)}m(A) \nonumber\\
		&\stackrel{\eqref{e:csj2}}{=}
		\Gamma^{(\rho)}(\psi,\psi)
		\bigl(A\cap
		(B(x,R+r/2+4\rho)\setminus B(x,R+r/2-\rho))\bigr)
		+\frac{C_2}{\phi(r)}m(A) \nonumber\\
		&\stackrel{\eqref{e:csj3}}{=}
		\Gamma^{(\rho)}(\wh\psi,\wh\psi)
		\bigl(A\cap
		(B(x,R+r/2+4\rho)\setminus B(x,R+r/2-\rho))\bigr)
		+\frac{C_2}{\phi(r)}m(A) \nonumber\\
		&\stackrel{\eqref{e:nlsi1}}{\le}
		\sum_{i\in\wh I}\Gamma^{(\rho)}(\psi_i,\psi_i)
		\bigl(A\cap
		(B(x,R+r/2+4\rho)\setminus B(x,R+r/2-\rho))\bigr)
		+\frac{C_2}{\phi(r)}m(A) \nonumber\\
		&\le \sum_{i\in\wh I}
		\Gamma^{(\rho)}(\psi_i,\psi_i)(A)
		+\frac{C_2}{\phi(r)}m(A).
	\end{align}
	
	Hence, for any $f\in\mathcal F\cap C(X)$,
	\begin{align} \label{e:csj5}
		\MoveEqLeft{\int_{B(x,R+r)} f^2\,d\Gamma(\psi,\psi)}
		\nonumber\\
		&\le \frac{C_2}{\phi(r)}
		\int_{B(x,R+r)}f^2\,dm
		+\sum_{i\in\wh I}
		\int_X f^2\,d\Gamma^{(\rho)}(\psi_i,\psi_i)
		\quad\mbox{(by \eqref{e:nlsi3})}
		\nonumber\\
		&\le \frac{C_2}{\phi(r)}
		\int_{B(x,R+r)}f^2\,dm
		+\sum_{i\in\wh I}\left(
		C_1\int_{B_i^\#\times B_i^\#}
		(f(y)-f(z))^2\,J(dy,dz)
		+\frac{C_1}{\phi(\rho)}
		\int_{B_i^\#}f^2\,dm
		\right).
	\end{align}
	By \eqref{e:overlap} and \eqref{e:csj4}, we have
	\begin{equation} \label{e:csj6}
		\sum_{i\in\wh I}\one_{B_i^\#}(y)
		\lesssim\one_{B(x,R+r)}(y),
	\end{equation}
	and
	\begin{equation} \label{e:csj7}
		\sum_{i\in\wh I}
		\one_{B_i^\#}(y)\one_{B_i^\#}(z)
		\lesssim\one_V(y)\one_V(z),
		\quad\mbox{where }
		V:=B(x,R+3r/4)\setminus B(x,R+r/4).
	\end{equation}
	Combining \eqref{e:csj5}, \eqref{e:csj6},
	\eqref{e:csj7}, and \eqref{e:scale}, for all
	$f\in\mathcal F\cap C(X)$ we have
	\[
	\int_{B(x,R+r)}f^2\,d\Gamma(\psi,\psi)
	\lesssim
	\frac{1}{\phi(r)}\int_{B(x,R+r)}f^2\,dm
	+\int_V\int_V
	(f(y)-f(z))^2J(y,z)\,m(dy)\,m(dz).
	\]
	By regularity of the Dirichlet form, the above estimate also
	holds for all $f\in\mathcal F$, and hence
	\hyperlink{csj}{$\operatorname{CSJ}(\phi)$} holds.
\end{proof}

\subsection{Characterization of stable-like heat kernel estimate and an application}
Next, we complete the proof of Theorem~\ref{t:char-hke} using Theorem~\ref{t:main}. Indeed, if $\diam(X,d)=\infty$ and $(X,d,m)$ satisfies the stronger  \hyperlink{rvd}{$\on{RVD}$} the equivalence is an immediate consequence of Theorem \ref{t:main}, \eqref{e:csj-cap} and \cite[Theorem~1.13]{CKW-hke}. In the general case, we need some additional ingredients from \cite{Mal, GHH-mv, GHH-lb, KM}.

\begin{proof}[Proof of Theorem~\ref{t:char-hke}]
	The implication \emph{(b) $\Rightarrow$ (a)} is well known and follows from \cite[Theorem~1.20]{Mal} together with \eqref{e:csj-cap} in the case $\diam(X,d)=\infty$.
	When $\diam(X,d)<\infty$, the same reasoning applies: the argument used to show \emph{(1) $\Rightarrow$ (4)} in \cite[Theorem~1.13]{CKW-hke}, as well as the construction and properties of the auxiliary space in \cite[\textsection~3]{Mal}, extend to the finite diameter case with only minor modifications.

	Next, we show (a) implies (b). By Theorem \ref{t:main}, the Dirichlet form $(\mathcal{E},\mathcal{F})$ satisfies \hyperlink{csj}{$\operatorname{CSJ}(\phi)$} and \hyperlink{jphi}{$\operatorname{J}(\phi)$}.  When $\diam(X,d)=\infty$, the stable-like heat kernel estimate \hyperlink{shk}{$\on{HK}(\phi)$} follows from obtaining the a similar heat kernel estimate for the auxiliary space in \cite[\textsection 3]{Mal} using \cite[eq.~(4.13), Proposition 4.1]{Mal} and \cite[Theorem~1.13]{CKW-hke}. The  stable-like heat kernel estimate \hyperlink{shk}{$\on{HK}(\phi)$} for the  auxiliary space then implies the desired heat kernel estimate for $(\mathcal{E},\mathcal{F})$ using \cite[eq.~(4.6) and (4.7), Proposition 4.1]{Mal}.
	
	It remains to show (a) implies (b) when $\diam(X,d)<\infty$.
	By adapting the arguments in \cite{Mal} for the finite diameter case, and using the argument in the previous paragraph, it suffices to further assume in addition that $(X,d,m)$ satisfies the reverse volume doubling property \hyperlink{rvd}{$\on{RVD}$}.The proof of the implication (a) implies (b) in \cite[Theorem 1.13]{CKW-hke} relies crucially on the infinite diameter assumption to obtain on-diagonal upper bounds as explained in \cite[Remark 2.41]{KM}.
	In this case, the argument in \cite[\textsection 4.3]{CKW-hke} implies   two-sided bounds on exit time which in turn is sufficient to obtain \hyperlink{shk}{$\on{HK}(\phi)$} by using the equivalence in \cite[Theorem 2.40]{KM}. The proof in \cite[Theorem 2.40]{KM} relies on results obtained in \cite{GHH-mv,GHH-lb}.  We refer to  \cite[Remark 8.3]{CC24b} for a different approach to  \cite[Theorem 2.40]{KM}.
\end{proof}

As an application of our results, we show that reflected pure jump
Dirichlet forms on fat open sets inherit two-sided stable-like
heat kernel bounds\footnote{These sets are called
	\emph{Ahlfors regular domains} in \cite{CC26}. We use the term
	\emph{fat open sets}, following Schramm's terminology for planar
	sets \cite[p.~308]{Sch}, to distinguish this condition from
	the usual notion of Ahlfors regularity.}.
This gives an alternative proof of \cite[Theorem~2.8]{CC26}.
The preprint \cite{CC26} appeared on arXiv shortly after
the first version of the present paper.
Our proof works directly with the reflected form and does not
require an extension operator from $D$ to $X$.
\begin{theorem}\label{t:reflected-hk}
	Let $(X,d,m)$ satisfy $\operatorname{VD}$ and
	$\operatorname{QRVD}$, and suppose that the pure jump Dirichlet
	form $(\mathcal E,\mathcal F)$ satisfies $\operatorname{HK}(\phi)$.
	Let $D\subset X$ be a nonempty open set such that, for some $c_D>0$,
	\[
	m(B(x,r)\cap D)\ge c_D m(B(x,r)),
	\qquad x\in D,\quad 0<r<\operatorname{diam}(D).
	\]
	Put $Y=\overline D$, let $\mu=m|_D$ be viewed as a measure on $Y$,
	and define
	\[
	\overline{\mathcal F}
	=\left\{u\in L^2(Y,\mu):
	\int_{D\times D}(u(x)-u(y))^2\,J(dx,dy)<\infty\right\},
	\overline{\mathcal E}(u,u)
	=\int_{D\times D}(u(x)-u(y))^2\,J(dx,dy).
	\]
	Then $(\overline{\mathcal E},\overline{\mathcal F})$ is a regular
	Dirichlet form on $L^2(Y,\mu)$ and satisfies
	$\operatorname{HK}(\phi)$.
\end{theorem}

\begin{proof}
	By Theorem \ref{t:char-hke}, the ambient form satisfies
	$\operatorname{J(\phi)}$ and $\operatorname{Cap(\phi)_{\le}}$.
	Let
	\[
	\mathcal C=\{v|_Y:v\in\mathcal F\cap C_c(X)\},
	\qquad
	\mathcal H=\overline{\mathcal C}^{\,\overline{\mathcal E}_1},
	\quad
	\overline{\mathcal E}_1(u,u)
	=\overline{\mathcal E}(u,u)+\|u\|_{L^2(Y,\mu)}^2.
	\]
	Restriction decreases both energy and $L^2$ norm, so
	$v|_Y\in\mathcal H$ for every $v\in\mathcal F$.
	The algebra $\mathcal C$ separates points and vanishes nowhere.
	Ambient regularity and the Stone--Weierstrass theorem show that
	$\mathcal C$ is uniformly dense in $C_c(Y)$. Hence
	$(\overline{\mathcal E},\mathcal H)$ is a regular Dirichlet form.
	
	Volume doubling and the assumption on $D$ give
	$\mu(B(x,r)\cap Y)\asymp m(B(x,r))$ for $x\in Y, 0<r<\diam(Y,d)$. Consequently $(Y,d,\mu)$ satisfies
	$\operatorname{VD}$ and $\operatorname{QRVD}$, and the restricted
	jump kernel satisfies $\operatorname{J(\phi)}$. Restricting
	ambient capacity cutoffs gives $\operatorname{Cap(\phi)_{\le}}$
	for $(\overline{\mathcal E},\mathcal H)$. Proposition
	\ref{p:domain}, applied on $Y$, now yields
	$\mathcal H=\overline{\mathcal F}$. Thus the reflected form is
	regular, and a second application of Theorem
	\ref{t:char-hke} gives $\operatorname{HK}(\phi)$.
\end{proof}
\noindent \textbf{Acknowledgments.}
I am grateful to Sylvester Eriksson-Bique for many illuminating discussions on the Whitney blending technique
and for sharing an early version of his manuscript \cite{Eri}. I thank Yijun Song for a question that led to Remark \ref{r:covering}-(ii). \newline
\noindent \textbf{Declarations.} There is no conflict of interest. Our manuscript has no associated data.

\noindent Department of Mathematics, University of British Columbia,
Vancouver, BC V6T 1Z2, Canada. \\
mathav@math.ubc.ca 


\begin{thebibliography}{4}
	
	
	
	\bibitem[AB]{AB} \textsc{Andres,  S., Barlow, M.~T.} 
	\newblock Energy inequalities for cutoff functions and some applications. 
	{\em J. fur reine angewandte Math.} {\bf 699} (2015), 183--216.
	
	\bibitem[Ant]{Ant} \textsc{Antilla, R.} 
	\newblock An approach to sub-Gaussian heat kernel estimates via analysis on metric spaces. \arxiv{2509.04155}.
	
	\bibitem[BCLS]{BCLS} \textsc{Bakry, D., Coulhon, T.,  Ledoux, M.,  Saloff-Coste, L. }
	\newblock Sobolev inequalities in disguise. 
	{\em Indiana Univ. Math. J.} {\bf 44} (1995), no. 4, 1033--1074.
	
	
	\bibitem[BGK]{BGK} \textsc{Barlow, M.~T., Grigor'yan, A.,  Kumagai, T.}
	\newblock	Heat kernel upper bounds for jump processes and the first exit time.
	{\em J. Reine Angew. Math.} {\bf 626} (2009), 135--157.
	
	
	\bibitem[BB]{BB}
	\textsc{Barlow,  M.~T., Bass, R.~F.} 
	\newblock Stability of parabolic Harnack inequalities. 
	{\em Trans. Amer. Math. Soc.} {\bf 356} (2004) no. 4, 1501--1533.
	
	
	
	\bibitem[BBK]{BBK}
	\textsc{Barlow,  M.~T., Bass, R.~F., Kumagai, T.} 
	Stability of parabolic Harnack inequalities on metric measure spaces, 
	{\em J. Math. Soc. Japan} (2) {\bf 58} (2006), 485--519.  (correction in \arxiv{2001.06714})
	
	
	\bibitem[BCM]{BCM} \textsc{Barlow, M.~T., Chen, Z.-Q., Murugan, M.} Stability of EHI and regularity of MMD spaces. \arxiv{2008.05152}
	
	
	\bibitem[BM]{BM} \textsc{Barlow, M.~T.,  Murugan, M.}
	\newblock Stability of the elliptic Harnack inequality. 
	{\em Ann. of Math.} (2) {\bf 187} (2018), no. 3, 777--823.
	
	
	\bibitem[BH]{BH91}
	{\sc Bouleau, N. and Hirsch, F.}
	\newblock {\em Dirichlet forms and analysis on Wiener space.}
	\newblock De Gruyter Stud. Math., {\bf 14}, Walter de Gruyter \& Co., Berlin, 1991. x+325 pp.
	
	
	\bibitem[CC]{CC24b} {\sc Cao, S., Chen, Z.-Q.}
	\newblock Boundary trace theorems for symmetric reflected diffusions. Probab. Theory Relat. Fields (2025). \url{https://doi.org/10.1007/s00440-025-01458-6}
	
	\bibitem[CC26]{CC26}  {\sc Cao, S., Chen, Z.-Q.} \newblock Heat kernel for reflected jump diffusion on Ahlfors regular domains, arXiv:2602.10448
	
	
	\bibitem[CKW20a]{CKW-phi}
	\textsc{Chen, Z.-Q., Kumagai, T., Wang, J.}  
	\newblock Stability of parabolic Harnack inequalities for symmetric non-local Dirichlet forms. 
	{\em  J. Eur. Math. Soc. (JEMS)} {\bf 22} (2020), no. 11, 3747--3803.
	
	\bibitem[CKW21]{CKW-hke} 	\textsc{Chen, Z.-Q., Kumagai, T., Wang, J.}   
	\newblock Stability of heat kernel estimates for symmetric non-local Dirichlet forms. 
	{\em Mem. Amer. Math. Soc.} {\bf 271} (2021), no. 1330, v+89 pp. 
	
	\bibitem[CKW20b]{CKW-diff-jump} 		\textsc{Chen, Z.-Q., Kumagai, T., Wang, J.}  
	\newblock Heat kernel estimates and parabolic Harnack inequalities for symmetric Dirichlet forms. 
	{\em Adv. Math.} {\bf 374} (2020), 107269, 71 pp.
	
	
	\bibitem[Del]{Del} \textsc{Delmotte, T.}
	\newblock Parabolic Harnack inequality and estimates of Markov chains on graphs {\em Rev. Mat. Iberoamericana} {\bf 15} (1999), no. 1, 181--232.
	
	
	\bibitem[Eri]{Eri} {\sc Eriksson-Bique, S.} On the Resistance Conjecture, \arxiv{2602.05477} (2026)
	
	\bibitem[EM]{EM}  {\sc Eriksson-Bique, S., Murugan, M.}  
	\newblock
	On the energy image density conjecture of Bouleau and Hirsch. 	\arxiv{2510.13659} (2025).
	
	\bibitem[FOT]{FOT}
	{\sc Fukushima, M., Oshima, Y. and Takeda, M.}
	\newblock {\em Dirichlet forms and symmetric Markov processes.}
	\newblock 2nd rev.\ and ext.\ ed., de Gruyter Stud. Math., vol.~19, Walter de Gruyter \& Co., Berlin, 2011. 
	
	
	\bibitem[Gri]{Gri}   \textsc{Grigor'yan, A.} 
	\newblock The heat equation on noncompact Riemannian 
	manifolds. 
	(in Russian) {\it Matem. Sbornik. \bf 182} (1991), 55--87.
	(English transl.) {\it Math. USSR Sbornik \bf 72} (1992), 47--77.
	
	
	\bibitem[GHH18]{GHH18}	\textsc{Grigor'yan, A., Hu, E.,  Hu, J.}
	\newblock Two-sided estimates of heat kernels of jump type Dirichlet forms. 
	{\em Adv. Math.} {\bf 330} (2018), 433--515.
	
	\bibitem[GHH23]{GHH-mv} \textsc{Grigor'yan, A., Hu, E.,  Hu, J.} Parabolic mean value inequality and on-diagonal upper bound of the heat kernel on doubling spaces. {\em Math. Ann.} (2023).  
	
	\bibitem[GHH24]{GHH-lb} \textsc{Grigor'yan, A., Hu, E.,  Hu, J.} Mean value inequality and generalized capacity on doubling spaces,  
	{\em Pure and Applied Funct. Anal.}   {\bf 9} (2024), no. 1, 111--168.
	
	
	
	\bibitem[GHL14]{GHL14} 	\textsc{Grigor'yan, A., Hu, J., Lau, K.-S.} 
	\newblock Heat kernels on metric measure spaces. 
	Geometry and analysis of fractals, 147--207, 
	{\em Springer Proc. Math. Stat.}, {\bf 88}, Springer, Heidelberg, 2014.
	
	
	
	
	\bibitem[GHL15]{GHL}
	\textsc{Grigor'yan, A., Hu, J.} and \textsc{Lau, K.-S.} 
	Generalized capacity, Harnack inequality and heat kernels of Dirichlet forms on metric spaces.
	{\em J. Math. Soc. Japan} {\bf 67} (2015) 1485--1549.
	
	
	
	
	\bibitem[GS]{GS} \textsc{Gyrya, P., Saloff-Coste, L.} 
	Neumann and Dirichlet heat kernels in inner uniform domains. 
	{\em Ast\'erisque} No. {\bf 336 }(2011), viii+144 pp.
	
	\bibitem[Hei]{Hei}
	{\sc Heinonen, J.}
	\newblock {\em Lectures on Analysis on Metric Spaces}, Universitext. Springer-Verlag, New York, 2001. x+140 pp. 
	
	\bibitem[KM20]{KM-aop} 
	{\sc  Kajino, N. and Murugan, M. }
	On singularity of energy measures for symmetric diffusions with full off-diagonal heat kernel estimates. 
	{\em Ann. Probab.} {\bf 48} (2020), no. 6, 2920–2951.
	
	\bibitem[KM25+]{KM} 
	{\sc  Kajino, N. and Murugan, M. }
	\newblock Heat kernel estimates for boundary traces of reflected diffusions on uniform domains. {\em arXiv preprint arXiv:2312.08546v3
	}. 2025
	
	\bibitem[KS]{KS} 
	{\sc  Kajino, N. and Shimizu, R. }
	\newblock Contraction properties and differentiability of $p$-energy forms with applications to nonlinear potential theory on self-similar sets. {\em arXiv preprint arXiv:2404.13668}. 2025
	
	
	
	\bibitem[Mal]{Mal} \textsc{Malmquist, J.}
	\newblock Stability results for symmetric jump processes on metric measure spaces with atoms.
	{\em Potential Anal.} {\bf 59} (2023), no. 1, 167--235.
	
	\bibitem[Mos]{Mos} \textsc{Mosco, U.} 
	\newblock Composite media and asymptotic Dirichlet forms. 
	\emph{J.\ Funct.\ Anal.} {\bf 123} (1994), no. 2, 368--421.
	
	\bibitem[Mur20]{Mur20}  \textsc{Murugan, M.}
	\newblock	On the length of chains in a metric space,
	\emph{J.\ Funct.\ Anal.}  \textbf{279} (2020), no. 6, 108627, 18 pp.
	
	\bibitem[Mur23]{Mur23}  \textsc{Murugan, M.}
	A note on heat kernel estimates, resistance bounds and Poincar\'e inequality. {\em Asian J. Math.}
	Vol. {\bf 27}, No. 6, pp. 853–866, 2023.
	
	\bibitem[Mur24]{Mur24}  \textsc{Murugan, M.}
	\newblock Heat kernel for reflected diffusion and extension property on uniform domains. 
	{\em Probab. Theory Related Fields} {\bf 190} (2024), no. 1-2, 543–599.
	
	\bibitem[Mur26]{Mur26}  \textsc{Murugan, M.}
	\newblock Diffusions and random walks with prescribed sub-Gaussian heat kernel estimates.
	{\em Ann. Probab.} (to appear).
	
	\bibitem[MS19]{MS} \textsc{Murugan, M., Saloff-Coste, L.}
	\newblock Heat kernel estimates for anomalous heavy-tailed random walks. 
	{\em Ann. Inst. Henri Poincar\'e Probab. Stat.} {\bf 55} (2019), no. 2, 697--719.
	
	\bibitem[MS23]{MS23}  \textsc{Murugan, M., Saloff-Coste, L.} Harnack inequalities and Gaussian estimates for random walks on metric measure spaces,  
	{\em Electron. J. Probab.} {\bf 28} (2023), Paper No. 64, 81 pp.
	
	\bibitem[Sal]{Sal} \textsc{Saloff-Coste, L.} 
	\newblock A note on Poincaré, Sobolev, and Harnack inequalities.
	{\em Internat. Math. Res. Notices} 1992, no. {\bf 2}, 27--38.
	
	\bibitem[Sch]{Sch} \textsc{Schramm, O.}
	\newblock Transboundary extremal length. 
	{\em  J. Anal. Math.} {\bf 66} (1995), 307--329.
	
	\bibitem[Stu]{Stu}  \textsc{Sturm, K.-T.}
	\newblock nalysis on local Dirichlet spaces --- III. The parabolic Harnack inequality,
	\emph{J.\ Math.\ Pures Appl.}\ (9) \textbf{75} (1996), no.\ 3, 273--297.
	
	\bibitem[Yan]{Yan} \textsc{Yang, M.} 
	\newblock  
	Energy inequalities for cutoff functions of -energies on metric measures spaces, \arxiv{2507.08577}
	
	
\end{thebibliography}
\end{document}